\documentclass[12pt,a4paper,reqno]{amsart}

\usepackage{amsfonts,amsthm,amsmath,amssymb,amscd,mathrsfs}
\usepackage{graphics}
\usepackage{indentfirst}
\usepackage{cite}
\usepackage{bm, enumerate,xcolor}
\usepackage{enumitem}
\usepackage[dvips]{epsfig}
\usepackage{latexsym}
\usepackage{float}
\usepackage{mathtools}
\usepackage[hidelinks]{hyperref} 
\usepackage{appendix}
\usepackage{accents}

\usepackage{tikz}
\usetikzlibrary{calc, intersections, patterns, decorations.pathreplacing, decorations.markings, arrows.meta}
\usepackage{caption}

\usepackage [english]{babel}
\usepackage [autostyle, english = american]{csquotes}
\MakeOuterQuote{"}

\newtheorem{theorem}{Theorem}[section]
\newtheorem*{theorem*}{Theorem}
\newtheorem{remark}{Remark}[section]

\newtheorem{definition}{Definition}[section]
\newtheorem{lemma}[theorem]{Lemma}
\newtheorem{pro}[theorem]{Proposition}

\def\mbN{\mathbb{N}}
\def\mbR{\mathbb{R}}
\def\mbS{\mathbb{S}}
\def\mbT{\mathbb{T}}
\def\mbZ{\mathbb{Z}}

\def\bbf{\mathbf{f}}

\def\bv{\mathbf{v}}

\def\bn{\mathbf{n}}

\def\mcH{\mathcal{H}}

\def\msC{\mathscr{C}}

\def\msN{\mathscr{N}}
\def\msR{\mathscr{R}}

\def\bGm{\Gamma^{\textup{bdry}}}
\def\cGm{\Gamma^{\textup{cont}}}
\def\nGm{\Gamma^{\textup{non}}}

\def\Div{{\rm div}}

\def\bnu{\boldsymbol{\nu}}

\DeclareMathOperator{\dist}{dist}

\newcommand{\Rn}[1]{\expandafter{\romannumeral #1\relax}}

\def\XXint#1#2#3{{\setbox0=\hbox{$#1{#2#3}{\int}$ }
		\vcenter{\hbox{$#2#3$ }}\kern-.6\wd0}}



\begin{document}

\title[A selection principle for 2D steady Euler flows]{A selection principle for 2D steady Euler flows via the vanishing viscosity limit}

\author{Changfeng Gui}
\address{Department of Mathematics, Faculty of Science,  University  of  Macau,  Macao  SAR,  P. R. China}
\email{changfenggui@um.edu.mo}

\author{Chunjing Xie}
\address{School of Mathematical Sciences,  Ministry of Education Key Laboratory of Scientific and Engineering Computing, CMA-Shanghai, Shanghai Jiao Tong University, Shanghai 200240, China.}
\email{cjxie@sjtu.edu.cn}

\author{Huan Xu}
\address{Department of Mathematics, Faculty of Science,  University  of  Macau,  Macao  SAR,  P. R. China}
\email{hzx0016@auburn.edu}

\date{}

\begin{abstract}
The 2D Euler system, which governs inviscid incompressible fluid flow, can admit infinitely many steady solutions in a given domain with slip boundary conditions. To select physical classical solutions, we investigate the vanishing viscosity limits of the steady Navier-Stokes system. The vanishing viscosity limits in periodic strips or bounded connected domains are completely characterized, even when strong boundary layers may appear.

More precisely, we show that the only vanishing viscosity limits in a bounded connected domain are flows with constant vorticity.
The significance of this result is that the approximating Navier-Stokes solutions are not required to have nested closed streamlines, an essential assumption in the century-old Prandtl-Batchelor theorem.
For flows in an infinitely long strip, if the viscous velocity (but not the pressure) is periodic in the strip direction, 
we show that the only vanishing viscosity limits are constant flows, Couette flows, and Poiseuille flows. 
The proof relies on a delicate analysis of the streamlines for both viscous and inviscid flows, in which a key observation is that the set of chaotic streamlines for the Euler flow is null with respect to two-dimensional Lebesgue measure.
The second result depends not only on the first but also on a powerful rigidity theorem that any non-shear steady classical Euler flow in a periodic strip must have closed streamlines, established via an analysis of streamlines and a novel total curvature estimate.
\end{abstract}

\maketitle



\section{Introduction}
In this paper, we investigate the steady incompressible Euler equations in a planar domain,
\begin{eqnarray}\label{steady_Euler}
\left\{\begin{aligned}
&\bv\cdot\nabla\bv+\nabla P=0, & {\rm in}\ \Omega,\\
&\Div\,\bv=0, & {\rm in}\ \Omega,
\end{aligned}\right.
\end{eqnarray}
where $\bv=(v_1,v_2)$ is the velocity field, $P$ is the scalar pressure, and $\Omega$ is a planar domain with a smooth boundary $\partial\Omega$. 
For a solid, impermeable boundary, the flow satisfies the following slip boundary condition
\begin{align}\label{slip boundary condition}
\bv\cdot\bn=0\ \ {\rm on}\ \partial\Omega,
\end{align}
where $\bn$ denotes the outward unit normal to $\partial\Omega$.

The system \eqref{steady_Euler} itself is one of the most important nonlinear PDEs as a macroscopic model to describe the steady motion of an inviscid incompressible fluid. Furthermore, the solutions of \eqref{steady_Euler} can be regarded as the steady states of the unsteady Euler system
\begin{eqnarray}\label{unsteady_Euler}
\left\{\begin{aligned}
&\partial_t\bv+\bv\cdot\nabla\bv+\nabla P=0,\\
&\Div\,\bv=0,
\end{aligned}\right.
\end{eqnarray}
which plays a crucial role in understanding the motion of inviscid incompressible fluids. It is well-known that in two dimensions, system \eqref{unsteady_Euler} admits a unique global solution for smooth initial and boundary conditions \cite{Holder_1933,Wolibner_1933}. However, its long-time behavior remains a major open problem in fluid PDEs \cite{Shnirelman_P_2013,Sverak_note_2011}. Characterizing all steady Euler flows could be helpful in this endeavor because they represent the possible "end states" of unsteady flows. Furthermore, a compelling reason to study steady Euler flows is their relevance to flows at high Reynolds number, where viscosity is confined to the boundary and the bulk is effectively inviscid. Notably, Prandtl's pioneering boundary layer theory \cite{Prandtl_1904} was originally developed in the setting of two-dimensional (2D) steady flows.

The difficulty in characterizing solutions to \eqref{steady_Euler} lies in the fact that the system may admit an infinite-dimensional family of solutions under the boundary condition \eqref{slip boundary condition}. The simplest example is the parallel shear flow, given by $\bv=(v_1(x_2),0)$ coupled with constant pressure. Another canonical family of solutions is the circular flow of the form 
\[
\bv(x)=V(|x|)\frac{(-x_2,x_1)}{|x|}\quad 
\text{and}\quad  P(x)=\int_{1}^{|x|}\frac{V^2(r)}{r}\,dr
\]
for an arbitrary function $V$. Furthermore, the 2D autonomous semilinear elliptic equation is a subclass of \eqref{steady_Euler} in the following sense: if $u$ is a solution to the equation $\Delta u=F'(u)$ in a planar domain, then 
\begin{align*}
\bv=\nabla^\perp u\vcentcolon=(-\partial_{x_2}u,\partial_{x_1}u),\ 
P=F(u)-\frac12|\nabla u|^2
\end{align*}
solve \eqref{steady_Euler}. System \eqref{steady_Euler} also has a solution with compact support, whose stream function solves a non-autonomous semilinear elliptic equation yet fails to be locally radially symmetric \cite{EFR_arXiv_2024}.

The abundance of solutions to \eqref{steady_Euler} raises the fundamental problem of selecting those that are physically realizable. 
While recent results on nonlinear inviscid damping \cite{BM_IHES_2015,IJ_Acta_2023,MZ_Annals_2024} show that some specific shear flows survive the long-time evolution of the Euler equations, it remains unclear whether more steady states enjoy such stability. In this paper, we adopt another natural approach to select steady Euler flows by taking the vanishing viscosity limit of solutions to the steady Navier–Stokes equations.
In fact, the vanishing viscosity limit of the Navier–Stokes equations has also been proposed as a selection principle for weak solutions of the unsteady Euler equations, see \cite{BTW_Paris_2012} and references therein.

To be specific, let us consider the steady incompressible Navier-Stokes equations
\begin{eqnarray}\label{steady_NS}
\left\{\begin{aligned}
&-\nu\Delta\bv^\nu+\bv^\nu\cdot\nabla\bv^\nu+\nabla P^\nu=0, & {\rm in}\ \Omega,\\
&\Div\,\bv^\nu=0, & {\rm in}\ \Omega.
\end{aligned}\right.
\end{eqnarray}
Here, the parameter $\nu>0$ denotes the viscosity coefficient, or equivalently, the reciprocal of the Reynolds number.
In his seminal paper \cite{Prandtl_1904} in 1904, Prandtl  made two closely related discoveries concerning steady flows at high Reynolds numbers. 
On the one hand, he introduced boundary-layer theory, explaining how viscosity may remain relevant in a thin layer adjacent to a solid boundary while the bulk flow is effectively inviscid. 
This idea initiated the vast theory of Prandtl boundary layers and the analysis of vanishing viscosity limit for solutions of Navier-Stokes equations; see, for example, \cite{OS_book_1999,GN_APDE_2017,GM_ARMA_2019,GI_CPAM_2023, Schlichting, Gao_Zhang,IM_AIA_2021,IM_FMP_2026} and the references therein. 
On the other hand, Prandtl also observed that, in a 2D steady flow at high Reynolds number, the vorticity remains constant within the closed streamline region.
It is this second insight that is most directly related to the selection problem studied in the present paper.
Prandtl's constant-vorticity principle was rediscovered by Batchelor in 1956 \cite{Batchelor_JFM_1956} and is now known as the Prandtl-Batchelor theorem. 
The selection of the limiting vorticity value is addressed in \cite{Batchelor_JFM_1956,FL_conference_1956,Wood_JFM_1957,Riley_JEM_1981,Kim_PF_1999,KC_SIAM_2001,van_FDR_2007,DIN_ARMA_2024}.

Two further issues require attention regarding the Prandtl-Batchelor theorem. First,  from a mathematical point of view, the assumption of the Prandtl-Batchelor theorem that all streamlines are closed requires justification. This important issue is one of the major motivations of our work. Second,
the topology in which $\bv^\nu$ converges to $\bv$ needs to be addressed. Numerical evidence suggests that the vanishing viscosity limit under weak convergence may select weak solutions of the Euler system, such as vortex sheets \cite{Okamoto_JDDE_1996,KO_AML_2023}.  As a first step toward the selection of general Euler flows, we focus here on local strong convergence in the vanishing viscosity limit, leaving the study of weak convergence for future work.

Assuming appropriate boundary conditions on $\bv^\nu$ and $\bv$, together with the convergence of $\bv^\nu$ to $\bv$ in some strong sense, the flow $\bv^\nu$ would possess closed streamlines if $\bv$ does.
For a $C^1$ solenoidal velocity field $\bv$ in a bounded connected domain satisfying \eqref{slip boundary condition}, the Morse-Sard theorem together with the coarea formula (applied to the stream function of $\bv$) implies that almost every streamline is closed (see Lemma \ref{lemma_null}).
In general domains, a steady Euler flow (such as a parallel shear flow) can have no closed streamlines. However, the authors \cite{GXX_arXiv_2024} recently discovered that every non-shear steady Euler flow in certain planar domains is not too far from having closed streamlines. Precisely, it was established in \cite{GXX_arXiv_2024} that for any bounded Euler flow $\bv$ in $\mbR^2$ or the periodic strip $\mbT\times(-1,1)$, the mapping $\frac{\bv}{|\bv|}$ must be onto the unit circle $\mbS^1$, unless $\bv$ is a parallel shear flow. This rigidity theorem generalizes 2D De Giorgi's conjecture \cite{GG_MA_1998}, analogous to how Little Picard theorem generalizes Liouville theorem for entire function and Osserman-Xavier-Fujimoto theorem \cite{Osserman_Annals_1964,Xavier_Annals_1981,Fujimoto_JMSJ_1988} generalizes Bernstein's theorem \cite{Bernstein_MZ_1927} for minimal surfaces.

Nevertheless, a bounded non-shear steady Euler flow in $\mbR^2$ may have no stagnation points, and thus no closed streamlines \cite{DR_CVPDE_2025}. 
In the strip $\mbT\times(-1,1)$, however, more recent rigidity results support the existence of (contractible) closed streamlines in non-shear flows.
Hamel and Nadirashvili \cite{HN_CPAM_2017} proved that for a steady Euler flow strictly away from stagnation points in $\mbR\times(-1,1)$, its streamlines must foliate the entire strip, which further forces all streamlines to be straight. Subsequently, it was shown in \cite{DR_CVPDE_2025,GRXX_arXiv_2025} that there exist non-stagnant, non-shear Euler flows in $\mathbb{R}\times(-1,1)$ that have infinitely long streamlines making a 180-degree turn. However, such streamlines cannot exist for steady Euler flows in $\mathbb{T}\times(-1,1)$. Recently, the results by Hamel and Nadirashvili were generalized in \cite{LLSX_arXiv_2022} where the stagnation appears on the boundary of the strip, and in \cite{DN_arXiv_2024} to the periodic setting under the {\it a priori} assumption that the flow is laminar while allowing stagnation points.

In this work, we take a step forward and give an affirmative answer to the existence of contractible closed streamlines in non-shear steady Euler flows in $\mathbb{T}\times(-1,1)$. The strength of this rigidity theorem is already evident from the fact that even its weaker form (i.e., the rigidity theorem in \cite{GXX_arXiv_2024}) is sufficient to imply the structural stability of shear flows with a convex profile \cite{GXX_arXiv_2024}.
Here, this enhanced rigidity theorem will play a pivotal role in the complete classification of vanishing viscosity limits in $\mathbb{T}\times(-1,1)$.
Lastly, we refer the reader to \cite{CDG_CMP_2021,CEW_ARMA_2023,DER_arXiv_2025,EHSX_arXiv_2024,GPSY_Duke_2021,HN_ARMA_2019,HN_JEMS_2023,LZ_ARMA_2011,Ruiz_ARMA_2023,WZ_arXiv_2023} for more rigidity results of steady Euler flows and to \cite{DG_PAMS_2024,DGN_arXiv_2025} for the existence of closed streamlines in Arnold-stable Euler flows.


\subsection{Main results}
Let us first give the precise definition of a vanishing viscosity limit.

\begin{definition}\label{defn_vvl}
Let $\Omega$ be a smooth planar domain. 
Let $\bv\in C^{2}(\Bar{\Omega})$ and $P\in C^{2}(\Bar{\Omega})$ solve \eqref{steady_Euler}-\eqref{slip boundary condition}, and $\bv^\nu\in C^2(\Omega)$ and $P^\nu\in C^1(\Omega)$ be a sequence of solutions to \eqref{steady_NS}.
Then $\bv$ is said to be a vanishing viscosity limit of $\{\bv^\nu\}$ in $\Omega$, if $\bv^\nu$ converges to $\bv$ in $L^\infty_{\textup{loc}}\cap H^1_{\textup{loc}}(\Omega)$ as $\nu$ tends to zero.
\end{definition}

The main results of this paper can be stated as follows.
\begin{theorem}\label{main_thm}
Suppose that $\Omega$ is a smooth domain in $\mathbb{R}^2$ and that $\bv\in C^{2}(\Bar{\Omega})$ and $P\in C^{2}(\Bar{\Omega})$ solve \eqref{steady_Euler}-\eqref{slip boundary condition}. 
Let $\bv^\nu\in C^2(\Omega)\cap C(\Bar{\Omega})$ and $P^\nu\in C^1(\Omega)$ be a sequence of solutions to \eqref{steady_NS}.
\begin{enumerate}[label=\textup{(\Alph*)}]
\item \label{thm_A} Suppose $\Omega$ is a bounded connected domain of the form
\begin{align}\label{defn_Omega}
\Omega=\Omega_0\setminus\left(\bigcup_{j=1}^{J}\overline{\Omega_j} \right),\quad \overline{\Omega_j} \subset\Omega_0,\ j=1,2,\cdots,J,
\end{align}
where $J$ is a nonnegative integer; for $J=0$, $\Omega=\Omega_0$; for $J\geq 1$, each $\Omega_j$ ($j=0,1,\cdots,J$) is a simply-connected bounded domain enclosed by a smooth Jordan curve $\bGm_j$, and $\{\overline{\Omega_j}\}_{j=1}^J$ are pairwise disjoint.
If $\Omega$ is multiply connected (i.e., $J\ge1$), we further assume the zero-flux condition on each boundary component
\begin{align}\label{zero_flux_bdd}
\int_{\bGm_j}\bn\cdot\bv^\nu \,ds=0, \quad j=0,1,\cdots,J,
\end{align}
where $\bn$ is the outward unit normal to $\partial\Omega$.
Then $\bv$ is a vanishing viscosity limit of $\{\bv^\nu\}$ in $\Omega$ in the sense of Definition \ref{defn_vvl} if and only if its vorticity $\omega\coloneqq\partial_{x_1}v_2-\partial_{x_2}v_1$ is constant.

\item \label{thm_B} If $\Omega =\mathbb{R}\times (-1,1)$ and $\bv$ is periodic in $x_1$, then $\bv$ is either a parallel shear flow or possesses contractible closed streamlines.
    
\item \label{thm_C}  Suppose $\Omega =\mathbb{R}\times (-1,1)$. If for all $\nu>0$,  $\bv^\nu=(v_1^\nu,v_2^\nu)$ are periodic in $x_1$ with the same period $T>0$, and satisfy the zero-flux boundary condition
\begin{align}\label{zero_flux_strip}
\int_{0}^{T}v_2^\nu(x_1,-1)\,dx_1=0, 
\end{align}
then $\bv$ is a vanishing viscosity limit of $\{\bv^\nu\}$ in $\Omega$ in the sense of Definition \ref{defn_vvl} if and only if $\bv=(c_0+c_1 x_2+c_2 x_2^2,0)$ for some constants $c_0,c_1,$ and $c_2$.
\end{enumerate}
\end{theorem}

There are several remarks in order.

\begin{remark}
The "if" parts in both Part \ref{thm_A} and Part \ref{thm_C} are straightforward. In fact, any flow with constant vorticity is a solution to both the Euler system \eqref{steady_Euler} and the Navier-Stokes system \eqref{steady_NS}. In a strip, any Euler flow of the form $\bv=(c_0+c_1x_2+c_2x_2^2,0)$ is the limit of the Navier-Stokes solutions $\bv^\nu=\bv$, $P^\nu=2\nu c_2 x_1$.
\end{remark}

\begin{remark}
The boundary condition \eqref{zero_flux_bdd} is necessary and sufficient for the existence of a stream function for $\bv^\nu$ in a multiply connected domain, whereas \eqref{zero_flux_strip} is necessary and sufficient for the periodic velocity $\bv^\nu$ to admit a periodic stream function in the strip. It goes without saying that \eqref{zero_flux_bdd} and \eqref{zero_flux_strip} are implied by the slip boundary condition \eqref{slip boundary condition}. However, neither the zero-flux boundary condition for a bounded connected domain nor the slip boundary condition in a periodic strip alone can imply that $\bv^\nu$ has a closed streamline.
\end{remark}

\begin{remark}
That $\bv^\nu$ converges to $\bv$ only on compact subsets of $\Omega$ is a reasonable assumption under which a Prandtl boundary layer may appear.
Indeed, when $\Omega$ is a disk or a periodic strip, there exists a sequence of solutions to \eqref{steady_NS} satisfying the slip boundary condition (hence \eqref{zero_flux_bdd} or \eqref{zero_flux_strip}) and featuring strong boundary layers, which converges to Couette flow at least in $L^\infty_{\textup{loc}}\cap H^1_{\textup{loc}}$ (see \cite{FCLT_CMP_2023,FPZ_arXiv_2023}).
\end{remark}

\begin{remark}
More steady Euler solutions could be selected as vanishing viscosity limits if we modify Definition \ref{defn_vvl}.
First, this becomes possible when an external force of the form $\bbf^\nu=\nu \bbf$ (with fixed $\bbf$) is applied to the viscous flow; insightful theoretical and numerical results can be found in \cite{OS_JIAM_1993,Okamoto_JDDE_1996,KO_AML_2023,FCLT_AIM_2024} and the references therein.   
However, Parts \ref{thm_A} and \ref{thm_C} still hold if the external force satisfies $\bbf^\nu=o(\nu)$ (as $\nu\rightarrow0^+$) (see Propositions \ref{prop_A} and \ref{prop_B}).
Second, a broader class of Euler solutions may be selected if the regularity required in Definition \ref{defn_vvl} is relaxed (see, e.g., \cite{Okamoto_JDDE_1996,KO_AML_2023}). This issue is also intimately related to the 2D Leray's problem, whose resolution relies on the analysis of $H^1$-weak solutions to the Euler system that arise as blow-up limits of a scaled Navier–Stokes system in bounded domains \cite{KPR_Annals_2015}.
\end{remark}

\begin{remark}
Note that $P^\nu$ in Part \ref{thm_C} is not required to be periodic in $x_1$. If $P^\nu$ is also periodic in $x_1$, the proof of Part \ref{thm_A} would apply so that $\omega$ has to be a constant, and thus the flow profile in Part \ref{thm_C} reduces to $\bv=(c_0+c_1 x_2,0)$.

In contrast, suppose a steady Euler flow $\bv$ in $\mathbb{R} \times [-1,1]$ satisfies the slip boundary condition \eqref{slip boundary condition} and is periodic in $x_1$. Then Bernoulli’s law asserts that $\partial_{x_1}\big(\frac12 v_1^2 + P\big)(x_1,\pm1)=0$ for all $x_1 \in \mathbb{R}$. Hence the boundary pressure $P(x_1,\pm1)$ is also periodic. This together with the periodicity of $\nabla P$ further implies that $P$ is periodic in the entire strip.
\end{remark}

\begin{remark}
The periodicity assumption on $\bv$ in Part \ref{thm_B} cannot be dropped, as there exist many smooth, bounded, non-shear steady Euler flows with a strictly positive vertical velocity component in $\mbR\times(-1,1)$ (see \cite{GXX_arXiv_2024,DR_CVPDE_2025,GRXX_arXiv_2025}). 
\end{remark}

\begin{remark}
The regularity requirement for $\bv$ and $P$ in Part \ref{thm_B} can be relaxed. For example, it suffices to have $\bv \in C^1(\mbR\times[-1,1])$ and $P \in C^1(\mbR\times[-1,1])$ (see Section \ref{sec_rigidity}).
\end{remark}

\begin{remark}
A non-circular steady Euler flow in an annulus need not have a contractible closed streamline. For example, a flow in an annulus that is circular only within two disjoint, non-contractible, non-concentric subannuli and stagnant otherwise is a non-circular Euler flow without contractible closed streamlines. Such a flow is said to be locally circular and its stream function is locally radially symmetric (see Figure \ref{local_circular}). 
\end{remark}


\begin{figure}[htbp]
\centering
\begin{tikzpicture}[scale=0.9, flow/.style={-{Stealth[length=2mm]}, thick, draw=blue!70!black},]

\begin{scope}
    \fill[gray!20] (0,0) circle (2);
\end{scope}

\begin{scope}
    \fill[blue!10] (0.1,0.1) circle (1.7);
\end{scope}

\begin{scope}
    \fill[gray!20] (0.1,0.1) circle (1.4);
\end{scope}

\begin{scope}
    \fill[blue!10] (-0.1,-0.1) circle (0.8);
\end{scope}

\begin{scope}
    \fill[gray!20] (-0.1,-0.1) circle (0.59);
\end{scope}

\begin{scope}
    \fill[white] (0,0) circle (0.3);
\end{scope}
    
\draw[thick, black] (0,0) circle (2);
\draw[thick, black] (0,0) circle (0.3);

\begin{scope}[shift={(0.1,0.1)}]
    \draw[thick, black!70!blue, dashed] (0,0) circle (1.7);
    \draw[thick, black!70!blue, dashed] (0,0) circle (1.4);
\end{scope}

\begin{scope}[shift={(0.1,0.1)}]
    \foreach \r in {1.4, 1.5, 1.6, 1.7} {
        \draw[flow, decoration={markings, mark=at position 0.2 with {\arrow{Stealth[length=1.5mm]}}}, postaction={decorate}] 
        (0,0) circle (\r);
    }
\end{scope}

\begin{scope}[shift={(-0.1,-0.1)}]
    \draw[thick, black!50!blue] (0,0) circle (0.8);
    \draw[thick, black!50!blue] (0,0) circle (0.59);
\end{scope}

\begin{scope}[shift={(-0.1,-0.1)}]
    \foreach \r in {0.59, 0.66, 0.73, 0.8} {
        \draw[flow, decoration={markings, mark=at position 0.2 with {\arrow{Stealth[length=1.5mm]}}}, postaction={decorate}] 
        (0,0) circle (\r);
    }
\end{scope}
\end{tikzpicture}

\caption{Locally circular flow without contractible closed streamlines in an annulus, the grey region consists of stagnation points.}
\label{local_circular}
\end{figure}


\subsection{Key ideas in the proof of the main results}
While Hamel and Nadirashvili \cite{HN_CPAM_2017,HN_ARMA_2019,HN_JEMS_2023} demonstrated the rigidity of 2D steady Euler flows either in the absence of stagnation points or when the stagnation set has a simple structure, the analysis becomes significantly more challenging when general stagnation points are present, as streamline behavior near them can be highly complex. Our approach is based on a key observation that the set of these chaotic streamlines is "small".

To be precise, let us introduce the stream function $u$ of $\bv$, so that $\bv=\nabla^\perp u$. By the slip boundary condition \eqref{slip boundary condition}, $u$ is constant on every component of the physical boundary. Denote by $\msC(u)$ and $\msR(u)$ the sets of critical values and regular values of $u$, respectively, i.e.,
\[
\msC(u)=\{u(x):x\in\Bar{\Omega}, |\nabla u(x)|=0 \},\quad
\msR(u)=\{u(x): x\in\Bar{\Omega}\}\setminus\msC(u).
\]
Then we decompose the domain into the disjoint union
\begin{align}\label{starting_point}
\Bar{\Omega}=u^{-1}(\msR(u))\cup \left(u^{-1}(\msC(u))\cap \{|\nabla u|>0 \}\right) \cup \{|\nabla u|=0 \}.
\end{align}
The Morse-Sard theorem and the coarea formula imply that $u^{-1}(\msC(u))\cap \{|\nabla u|>0 \}$ is an $\mcH^2$-null subset of $\mbR^2$. As a consequence, our primary analysis will focus on the set $u^{-1}(\msR(u))$.

\subsubsection{Prandtl-Batchelor theorem in a bounded connected domain}
By \eqref{starting_point}, we see that $\nabla\omega=0$ $\mcH^2$-a.e. in $\Omega\setminus u^{-1}(\msR(u))$.
Theorem \ref{main_thm}-\ref{thm_A} will follow once we prove that $\nabla\omega=0$ in $u^{-1}(\msR(u))$. Given that $u$ is constant on every component of $\partial\Omega$, the level set $u^{-1}(r)$ for every $r\in\msR(u)\setminus\{u|_{\bGm_j}:0\le j\le J \}$ consists of finitely many Jordan curves in $\Omega$. Let $\Gamma$ be any such Jordan curve. One can find a closed streamline $\Gamma^\nu$ of $\bv^\nu$ arbitrarily close to $\Gamma$. Now taking the dot product between $\frac{\bv^\nu}{|\bv^\nu|}$ and the momentum equation of \eqref{steady_NS} yields
\begin{align*}
\frac{\bv^\nu}{|\bv^\nu|}\cdot\nabla \left(\frac12|\bv^\nu|^2+P^\nu\right)
=\nu \frac{\bv^\nu}{|\bv^\nu|}\cdot\nabla^\perp\omega^\nu,   
\end{align*}
where $\omega^\nu$ denotes the vorticity of $\bv^\nu$. Next, integrating the resulting equation along $\Gamma^\nu$ gives
\begin{align}\label{integral_constraint_epsl}
\int_{\Gamma^\nu}\frac{\bv^\nu}{|\bv^\nu|}\cdot\nabla^\perp\omega^\nu \,ds
=0.   
\end{align}
Formally, passing to the limit as $\nu \rightarrow 0^+$ leads to
\begin{align}\label{integral_constraint}
\int_{\Gamma}\frac{\bv}{|\bv|}\cdot\nabla^\perp \omega\,ds=0.
\end{align}
This, along with the vorticity equation
\begin{equation}\label{Euler_vorticity}
\bv\cdot\nabla\omega=0,
\end{equation}
implies that $\nabla\omega=0$ in $u^{-1}(\msR(u))$. 
However, the convergence from \eqref{integral_constraint_epsl} to \eqref{integral_constraint} is nontrivial under the assumption that $\bv^\nu\rightarrow \bv$ merely in $L^\infty_{\textup{loc}}\cap H^1_{\textup{loc}}$.
We will justify this in Section \ref{sec_selection_bdd} by integrating the left-hand side of \eqref{integral_constraint_epsl} twice with respect to the level value of the stream function of $\bv^\nu$, which yields an integral that converges under the assumption on the $L^\infty_{\textup{loc}}\cap H^1_{\textup{loc}}$ convergence.

\subsubsection{Total curvature decomposition} To prove Theorem \ref{main_thm}-\ref{thm_B}, we show that if $\bv$ does not have any contractible closed streamlines, then the {\it total curvature} is trivial, i.e., 
\begin{align}\label{trivial_total_curvature}
\iint_{\mbT\times(-1,1)}(|\nabla\bv|^2-|\nabla|\bv||^2)\,dx=0.
\end{align}
The name of the above integral is attributed to the fact that
\begin{align*}
|\nabla\bv|^2-|\nabla|\bv||^2=|\bv|^2(\kappa^2+\tau^2), 
\end{align*}
where $\kappa$ is the curvature of the streamline and $\tau$ is the curvature of the curve perpendicular to the streamline. The total curvature can be estimated by reformulating \eqref{Euler_vorticity} as
\begin{equation}\label{Euler_divergence_form}
\Div(v_1\nabla v_2-v_2\nabla v_1)=0.    
\end{equation}
If $\bv$ admits a phase function $\theta$ with $\frac{\bv}{|\bv|}=(\cos\theta,\sin\theta)$, it follows from \eqref{Euler_divergence_form} that $\theta$ solves
\begin{align}\label{elliptic_phase}
\Div(|\bv|^2\nabla\theta)=0.    
\end{align}
Hamel and Nadirashvili \cite{HN_ARMA_2019} used the above equation as the starting point to prove their rigidity theorem for steady Euler flows in $\mbR^2$ that are strictly away from stagnation.
In previous work \cite{GXX_arXiv_2024}, under the assumption that $\bv/|\bv|$ omits a value in $\mbS^1$, the authors derived an approximating equation for \eqref{elliptic_phase}. This led to estimates for the total curvature and, consequently, to rigidity results for steady Euler flows in certain unbounded domains.

However, equation \eqref{elliptic_phase} may fail to hold globally under the sole assumption of no contractible closed streamlines. Our key observation is that under this assumption, $u^{-1}(\msR(u))$ can be decomposed into at most countably many curved periodic strips $\{S_j\}$, in each of which $\bv$ admits a phase $\theta_j$ satisfying \eqref{elliptic_phase}. Applying energy estimates to \eqref{elliptic_phase} and Bernoulli’s law, we find that
\begin{align*}
\iint_{S_j}(|\nabla \bv|^2-|\nabla|\bv||^2)\,dx
=-\frac12\iint_{S_j}\Delta P\,dx.
\end{align*}
Since $u^{-1}(\msC(u))\cap \{|\nabla u|>0 \}$ is $\mcH^2$-null, and 
\begin{align*}
|\nabla \bv|^2-|\nabla|\bv||^2=0 \ \ {\rm and} \ \ \Delta P=0,\quad \mcH^2\textup{-a.e.} \ {\rm in}\    \{|\nabla u|=0 \},
\end{align*}
it holds that
\begin{align*}
\iint_{\mbT\times(-1,1)}(|\nabla \bv|^2-|\nabla|\bv||^2)\,dx
=-\frac12\iint_{\mbT\times(-1,1)}\Delta P\,dx.
\end{align*}
Finally, it follows from \eqref{steady_Euler}-\eqref{slip boundary condition} that the integral on the right vanishes, and thus \eqref{trivial_total_curvature} holds.

\subsubsection{Propagation of constant vorticity} 
Under the assumptions of Theorem \ref{main_thm}-\ref{thm_C}, analogous to \eqref{integral_constraint}, there exists a constant $b$ such that 
\begin{align}\label{line_integral_b}
\int_{\nGm}\bn\cdot\nabla \omega\,ds=b,
\end{align}
where $\nGm$ is any non-contractible closed streamline of $\bv$, and $\bn$ is the upward-pointing unit normal to $\nGm$. 
The difficulty is that $b$ is not necessarily zero because $P^\nu$ is not necessarily periodic in $x_1$.

If $\bv$ does not have any contractible closed streamline, it follows from Theorem \ref{main_thm}-\ref{thm_B} that $\bv$ must be a shear flow. Then it is easy to get from \eqref{line_integral_b} that $\bv=(c_0+c_1 x_2+c_2 x_2^2,0)$ for some constants $c_0,c_1,$ and $c_2$.

If $\bv$ has a contractible closed streamline $\cGm$, we infer from Theorem \ref{main_thm}-\ref{thm_A} that $\omega$ is constant in $\textup{Int}(\cGm)$. 
This region can be extended to a maximal constant-vorticity domain $D$, whose boundary is approximated by a sequence of contractible closed streamlines inside $D$. Then by Hopf's lemma, there exists a point $z\in\partial D\cap (\mbT\times(-1,1))$ such that $|\bv(z)|>0$. We can show that the constancy of $\omega$ in $D$ propagates across $z$ to the entire strip. Indeed, for any sufficiently small disk $B_{\delta_0}(z)$ centered at $z$ with radius $\delta_0>0$, it follows from \eqref{starting_point} that almost all points of $B_{\delta_0}(z)\setminus\Bar{D}$ belong to $u^{-1}(\msR(u))$. By the maximality of $D$, $B_{\delta_0}(z)\setminus\Bar{D}$ cannot contain any point on a contractible closed streamline. Hence, any small neighborhood of $z$ must contain a point lying on a non-contractible closed streamline. This forces the constant $b$ in \eqref{line_integral_b} to be zero, implying that $\omega$ is constant in the entire strip.
Finally, the only constant-vorticity flow in a periodic strip is the shear flow $(c_0+c_1 x_2,0)$, which contradicts the existence of a contractible closed streamline.

\subsection{Organization of the paper} 
The rest of the paper is organized as follows. Section \ref{sec_pre} collects the necessary preliminaries. Since the proofs of Parts \ref{thm_A} and \ref{thm_C} of Theorem \ref{main_thm} share many of the same techniques, while the proof of Part \ref{thm_C} relies on both Parts \ref{thm_A} and \ref{thm_B} (and Part \ref{thm_B} is independent of Part \ref{thm_A}), we first prove Part \ref{thm_B} in Section \ref{sec_rigidity}. The proofs of Parts \ref{thm_A} and \ref{thm_C} are then given in Sections \ref{sec_selection_bdd} and \ref{sec_selection_strip}, respectively.

\subsection{Notation}
Throughout this paper, $\mbT$ denotes the 1-torus $\mbR/2\pi\mbZ$.
For a vector $(a,b)$, define $(a,b)^\perp=(-b,a)$, and for a scalar function $u$, $\nabla^\perp u=(\nabla u)^\perp$. If $u$ is defined on a domain $\Omega$, $u(\Omega)$ denotes its range; for a subset $I\subset u(\Omega)$, $u^{-1}(I)$ denotes the preimage of $I$, and if $r\in u(\Omega)$, $u^{-1}(r)=\{x\in\Omega:u(x)=r\}$ is the $r$-level set of $u$.
Let $B_{\delta}(z)$ be the open disk of radius $\delta$ centered at $z$. For a Jordan curve $\Gamma$ on the plane, $\textup{Int}(\Gamma)$ means the bounded connected component of $\mbR^2\setminus\Gamma$.
Finally, $\mcH^n$ is the $n$-dimensional Hausdorff measure.


\medskip

\section{Preliminaries}\label{sec_pre}
This section is devoted to the analysis of the streamlines for both viscous and inviscid flows. 

Suppose $\Omega$ is a planar open domain.
For a function $u\in C^2(\bar{\Omega})$, denote by $\msC(u)$ the set of critical values of $u,$ and $\msR(u)$ the set of regular values of $u$; that is,
\[
\msC(u)=\{u(x):x\in\Bar{\Omega}, |\nabla u(x)|=0 \},\quad
\msR(u)=u(\Bar{\Omega})\setminus\msC(u).
\]
The Morse-Sard theorem says that $\msC(u)$ is an $\mcH^1$-null subset of $\mbR$. Although $u^{-1}(\msC(u))$ may not be small, it is well-known that $u^{-1}(\msC(u))\cap \{|\nabla u|>0 \}$ is an $\mcH^2$-null subset of $\mbR^2$.
This fact plays a crucial role in the proof of Theorem \ref{main_thm}. We shall provide the proof for the reader's convenience.

\begin{lemma}\label{lemma_null}
For every $\mcH^1$-null subset $\msN$ of $u(\Bar{\Omega})$, $u^{-1}(\msN)\cap\{|\nabla u|>0 \}$ is an $\mcH^2$-null subset of $\Bar{\Omega}$. In particular, $u^{-1}(\msC(u))\cap\{|\nabla u|>0 \}$ is an $\mcH^2$-null subset of $\Bar{\Omega}$.
\end{lemma}
\begin{proof}
We write
\begin{align*}
u^{-1}(\msN)\cap\{|\nabla u|>0 \}
=\bigcup_{k\in\mbN_+} u^{-1}(\msN)\cap\left\{|\nabla u|>\frac{1}{k} \right\}\cap B_{k},
\end{align*}
where $B_k =\{x\in\mbR^2:|x|<k\}$. 
It follows from the coarea formula (\hspace{1sp}\cite{EvansG}) that
\begin{equation*}
\begin{aligned}
\mcH^2\left(u^{-1}(\msN)\cap\{|\nabla u|>\frac{1}{k} \}\cap B_{k}\right)= &
\iint_{u^{-1}(\msN)\cap\{|\nabla u|>\frac{1}{k} \}\cap B_{k}} 1\,dx\\
= &\int_{\msN}\int_{u^{-1}(s)\cap\{|\nabla u|>\frac{1}{k} \}\cap B_{k}}\frac{1}{|\nabla u|} \,d\mcH^1\,ds
=0.
\end{aligned}
\end{equation*}
The first statement of the lemma follows directly. Since $\msC(u)$ is $\mcH^1$-null by the Morse-Sard theorem, the second statement follows immediately as well.
\end{proof}

In the rest of this section, we address the existence of closed streamlines for the viscous velocity field $\bv^\nu$, a prerequisite for applying the Prandtl–Batchelor theory to characterize the vanishing viscosity limits. Assuming that the inviscid velocity field $\bv$ possesses a closed streamline $\Gamma$, we show that for sufficiently small $\nu$ the viscous field $\bv^\nu$ also admits closed streamlines near $\Gamma$, provided that $\bv^\nu$ admits a stream function (which is guaranteed by the zero-flux boundary condition \eqref{zero_flux_bdd} or \eqref{zero_flux_strip}) and that $\bv^\nu$ converges to $\bv$ in $L^\infty_{\text{loc}}$ as $\nu\rightarrow0^+$.

For the remainder of this section, $\Omega$ denotes a bounded connected domain as defined in \eqref{defn_Omega}. Nevertheless, the results presented here extend analogously to the case of the periodic strip $\mbT\times(-1,1)$.

\begin{lemma}\label{prelim_u_convergence}
Let $\bv, \bv^\nu\in C^1(\Omega)\cap C(\Bar{\Omega})$ be solenoidal vector fields, both satisfying the zero-flux boundary condition \eqref{zero_flux_bdd}, such that $\bv^\nu\rightarrow\bv$ in $L^\infty_{\textup{loc}}(\Omega)$ as $\nu\rightarrow0^+$. Let $u\in C^2(\Omega)$ be a stream function of $\bv$. Then there exist stream functions $u^\nu$ of $\bv^\nu$ such that $u^\nu\rightarrow u$ in $C^1_{\textup{loc}}(\Omega)$ as $\nu\rightarrow0^+$.
\end{lemma}
\begin{proof}
The existence of globally defined stream functions for $\bv$ and $\bv^\nu$ is guaranteed by the zero-flux boundary condition \eqref{zero_flux_bdd}.
Fix $z\in \Omega$ and let $u^\nu$ be the unique stream function of $\bv^\nu$ such that $u^\nu(z)=u(z).$ Given any subdomain $\Omega_1\Subset\Omega$ with $z\in \Omega_1$, choose a connected subdomain $\Omega_2$ such that $\Omega_1\Subset\Omega_2\Subset\Omega$. Then there exists a positive integer $m=m(\Omega_2)$ such that $\overline{\Omega_1}$ is covered by $m$ open disks, each of which is contained in $\Omega_2$. Now for any $x\in\Omega_1$, we can connect $x$ and $z$ with a continuous path that consists of at most $m$ line segments, each of which is contained in a covering disk. List the endpoints of the line segments in order: $x^0=x,x^1,\cdots,x^{n-1},x^n=z$, $n\le m.$ Since $(u^\nu-u)(x_n)=0$ and $|\nabla u^\nu-\nabla u|=|\bv^\nu-\bv|$, we have
\begin{align*}
|u^\nu(x)-u(x)|
=&\left| \sum_{j=1}^{n} [(u^\nu-u)(x^{j-1})
-(u^\nu-u)(x^j)]\right|\\
\le& m\cdot \textup{Diam}(\Omega_2)\cdot\|\bv^\nu-\bv\|_{L^\infty(\Omega_2)},
\end{align*}
where $\textup{Diam}(\Omega_2)$ denotes the diameter of $\Omega_2$. It immediately follows that $u^\nu\rightarrow u$ in $C^1(\overline{\Omega_1})$ as $\nu\rightarrow0^+$.
This proves the lemma.
\end{proof}

Let $u\in C^2(\Bar{\Omega})$.
Suppose, for some $r\in u(\Bar{\Omega})$, $\Gamma_r\subset u^{-1}(r)$ is a $C^2$ Jordan curve inside $\Omega$ on which $|\nabla u|>0$.
Let $0<d<\dist(\Gamma_r,\partial\Omega)$, and we will impose further restrictions on $d$ as needed. 
Let $\bn$ be the outward unit normal to $\Gamma_r$.
For $x\in\Gamma_r,$ define
\[
I(x,d)=\{x+t\bn(x):-d\le t\le d\}; \quad N(\Gamma_r,d)=\cup_{x\in\Gamma_r}I(x,d).
\]
Let $d$ be smaller than the normal injectivity radius of $\Gamma_r,$ so that $\{ I(x,d) \}_{x\in\Gamma_r}$ are disjoint, and $N(\Gamma_r,d)$ is a tubular neighborhood of $\Gamma_r$. For notational convenience, we define
\[
A_{t_1,t_2}(r)=\{r+t_1\le u\le r+t_2 \}\cap N(\Gamma_r,d),
\]
and when no ambiguity arises, we simply write $A_{t_1,t_2}=A_{t_1,t_2}(r)$.
Without loss of generality, we may assume that
\begin{align}\label{assump_dnu}
\bn\cdot\nabla u=|\nabla u|>0 \quad {\rm on}\ \Gamma_r.
\end{align}

\begin{lemma}\label{prelim_lemma}
Suppose $u\in C^2(\Bar{\Omega})$ and $u^\nu\in C^2(\Omega)$ satisfy that $u^\nu\rightarrow u$ in $C^1_{\textup{loc}}(\Omega)$ as $\nu\rightarrow0^+$. 
If $d$ satisfies
\begin{align}\label{small_d}
d<\frac{\min_{\Gamma_r}|\nabla u|}{2\|\nabla^2 u\|_{L^\infty(\Bar{\Omega})}},
\end{align}
then there exists a number $\delta_0>0$ such that the following statements hold.

\begin{enumerate}[label=\textup{(\roman*)}]
\item $A_{-2\delta_0,2\delta_0}$ is a tubular neighborhood of $\Gamma_r$ foliated by the level sets of $u$, that is, $A_{-2\delta_0,2\delta_0}=\bigcup_{|t|\le 2\delta_0}\Gamma_{r+t}$, where every $\Gamma_{r+t}\subset u^{-1}(r+t)$ is a $C^2$ Jordan curve, and $\Gamma_{r+t_1}\subset \textup{Int}(\Gamma_{r+t_2})$ whenever $-2\delta_0\le t_1<t_2\le 2\delta_0$.

\item For sufficiently small $\nu$ and for every $t\in[-\delta_0,\delta_0]$, $\Gamma_{r+t}^\nu\vcentcolon=(u^\nu)^{-1}(r+t)\cap N(\Gamma_r,d)$ is a $C^2$ Jordan curve and $A^\nu_{-\delta_0,\delta_0}$ is foliated by $\{ \Gamma^\nu_{r+t} \}_{|t|\le \delta_0}$, where 
\[
A^\nu_{t_1,t_2}=A^\nu_{t_1,t_2}(r)\coloneqq\{r+t_1\le u^\nu\le r+t_2 \}\cap N(\Gamma_r,d), \ for \ -\delta_0\le t_1\le t_2\le \delta_0.
\]

\item For $-\delta_0\le t_1< t_2\le \delta_0$, it holds that
\begin{align*}
\lim_{\nu\rightarrow 0}\mcH^2\left(A_{t_1,t_2}\triangle A^\nu_{t_1,t_2} \right)=0,
\end{align*}
where $\triangle$ denotes the symmetric difference of two sets, i.e.,
\[
A_{t_1,t_2}\triangle A^\nu_{t_1,t_2}
=\left( A_{t_1,t_2}\setminus A^\nu_{t_1,t_2} \right)\cup \left( A^\nu_{t_1,t_2}\setminus A_{t_1,t_2} \right).
\]
\end{enumerate}
\end{lemma}
\begin{proof}
(\Rn{1}) 
For every $x\in \Gamma_r$ and every $y\in I(x,d),$ it follows from \eqref{assump_dnu} and \eqref{small_d} that
\begin{align}\label{normal_gradu}
\bn(x)\cdot\nabla u(y)\ge |\nabla u(x)|-\|\nabla^2 u\|_{L^\infty(\Bar{\Omega})}|y-x|
\ge \frac12 \min_{\Gamma_r}|\nabla u|>0.
\end{align}
This implies that $u$ is strictly monotone on $I(x,d)$ for every $x\in\Gamma_r$. As a consequence,
\begin{align*}
\delta_0\vcentcolon=\frac12 \min\left\{ r-\max_{x\in\Gamma_r}u(x-d\bn(x)),\min_{x\in\Gamma_r}u(x+d\bn(x))-r\right\}>0.
\end{align*}
With this choice of $\delta_0$, part (\Rn{1}) readily follows.  

(\Rn{2}) 
For sufficiently small $\nu,$ and for every $x\in\Gamma_r$ and $y\in I(x,d)$, by \eqref{normal_gradu} and the assumption that $u^\nu\rightarrow u$ in $C^1_{\textup{loc}}(\Omega)$, we get
\begin{align}\label{normal_graduep}
\bn(x)\cdot\nabla u^\nu(y)
\ge \bn(x)\cdot\nabla u(y)-\|\nabla u^\nu-\nabla u\|_{L^\infty(N(\Gamma_r,d))}
\ge \frac14\min_{\Gamma_r}|\nabla u|>0.
\end{align}
So $u^\nu$ is also strictly monotone on $I(x,d)$ for every $x\in\Gamma_r$.
On the other hand, for sufficiently small $\nu$ and for every $x\in \Gamma_{r-2\delta_0}$ we have
\[
u^\nu(x)-(r-2\delta_0)=u^\nu(x)-u(x)<\delta_0,
\]
which implies that $u^\nu< r-\delta_0$ on $\Gamma_{r-2\delta_0}$. Similarly, $u^\nu> r+\delta_0$ on $\Gamma_{r+2\delta_0}$. 
Thus, $\Gamma_{r+t}^\nu$ is a $C^2$ Jordan curve for every $t\in[-\delta_0,\delta_0],$ and $\Gamma_{r+t_1}^\nu\subset \text{Int}(\Gamma_{r+t_2}^\nu)$ whenever $-\delta_0\le t_1<t_2\le \delta_0$.
In conclusion, $A^\nu_{-\delta_0,\delta_0}$ is foliated by the level sets of $u^\nu$.

(\Rn{3})
Fix any positive number $\delta$ with $\delta<\min\{\delta_0,\frac12(t_2-t_1)\}$. Then for $\nu$ sufficiently small, it must hold that
\begin{align*}
A_{t_1+\delta,t_2-\delta} \subset A^\nu_{t_1,t_2}\subset A_{t_1-\delta,t_2+\delta}.
\end{align*}
Consequently, we have
\begin{align*}
A_{t_1,t_2}\triangle A^\nu_{t_1,t_2}\subset
A_{t_1-\delta,t_2+\delta}\setminus A_{t_1+\delta,t_2-\delta}.
\end{align*}
By the dominated convergence theorem, it is easy to see that
\begin{align*}
\lim_{\delta\rightarrow0}\mcH^2\left( A_{t_1-\delta,t_2+\delta}\setminus A_{t_1+\delta,t_2-\delta}\right)
=\lim_{\delta\rightarrow0}\iint_{A_{-2\delta_0,2\delta_0}} 1_{A_{t_1-\delta,t_2+\delta}\setminus A_{t_1+\delta,t_2-\delta}}\,dx=0.
\end{align*}
This completes the proof of the lemma.
\end{proof}


\medskip

\section{Rigidity via total curvature estimate}\label{sec_rigidity}

In this section, we establish Theorem \ref{main_thm}-\ref{thm_B} via a total curvature estimate for the flow.

The regularity requirement for $(\bv,P)$ in Theorem \ref{main_thm}-\ref{thm_B} can be relaxed.
In the sequel, let $\bv\in C^1(\mbT\times[-1,1])$, $P\in C^1(\mbR\times[-1,1])$ be a solution to \eqref{steady_Euler} satisfying $v_2(\cdot,\pm1)=0$. Let $u$ be the stream function of $\bv$ given by
\[
u(x_1,x_2)=\int_{x_2}^{1}v_1(x_1,t)\,dt.
\]
So $u\in C^2(\mbT\times[-1,1])$. Obviously, $u(\cdot,1)=0$ on $\mbT$. By the incompressibility condition and the slip boundary condition \eqref{slip boundary condition}, we infer that $u(\cdot,-1)= c$ on $\mbT$ for some constant $c$. 
Again, let $\msR(u)$ and $\msC(u)$ denote the sets of regular values and critical values of $u$, respectively. Then we can decompose the strip into the disjoint union
\begin{align}\label{decomposition_strip}
\mbT\times(-1,1)=u^{-1}(\msR(u))\cup \{|\nabla u|=0 \} \cup (u^{-1}(\msC(u))\cap \{|\nabla u|>0 \}),
\end{align}
where all the subsets on the right side are defined as subsets of $\mbT\times(-1,1)$.

For every $r\in\msR(u)$, the level set $u^{-1}(r)$ is a closed subset of $\mbT\times(-1,1)$. Here and below, we adopt the convention that, when $r=0$ or $r=c$, $u^{-1}(r)$ denotes only the portion of the level set lying inside $\mbT\times(-1,1)$. By the implicit function theorem, $u^{-1}(r)$ is a $C^2$ closed (compact and without boundary) one-dimensional manifold in $\mbT\times(-1,1)$, where $\mbT\times(-1,1)$ is regarded as the lateral surface of the cylinder $B_1\times(-1,1)$. Hence, by the classification of closed one-dimensional manifolds, $u^{-1}(r)$ consists of finitely many $C^2$ simple {\it closed} curves. Each such curve is either a contractible or a non-contractible closed streamline in $\mbT\times(-1,1)$.
By the Jordan curve theorem, every non-contractible closed streamline divides $\mbT\times(-1,1)$ into the upper component and the lower component. Therefore, we can define the position of a point or subset as being "above" or "below" a non-contractible closed streamline; or as being "between" two distinct non-contractible closed streamlines. 

For the sake of clarity, we provide the following definitions

\begin{definition}\label{defn_regular_tube}
An open subset $S$ of $\mbT\times(-1,1)$ is called a regular tube if it is bounded by two distinct non-contractible closed streamlines and foliated by non-contractible closed streamlines.
\end{definition}

\begin{definition}\label{defn_max_tube}
$M$ is called a maximal tube if there exists a sequence of regular tubes $\{S_k \}_{k\in\mbN_+}$ satisfying

\begin{enumerate}[label=\textup{(\roman*)}]
\item $M=\bigcup_{k\in \mbN_+} S_k$;

\item $S_k\Subset S_{k+1}$ for every $k\in \mbN_+$;

\item  If there exists a sequence of regular tubes $\{\Tilde{S}_k \}_{k\in\mbN_+}$ such that $\Tilde{S}_k\Subset \Tilde{S}_{k+1}$, $k\in \mbN_+$, and $M\subset \Tilde{M}\coloneqq\bigcup_{k\in \mbN_+} \Tilde{S}_k$, then $M=\Tilde{M}$.
\end{enumerate}

We call $\{S_k \}_{k\in\mbN_+}$ a sequence of approximating tubes for $M$.
\end{definition}

\begin{lemma}\label{lemma_maxtube}
Under Definitions \ref{defn_regular_tube} and \ref{defn_max_tube}, the following statements hold.
\begin{enumerate}[label=\textup{(\roman*)}]
\item Any two maximal tubes are either identical to each other or disjoint.

\item  For every non-contractible closed streamline $\Gamma$, there exists a unique maximal tube $M$ containing $\Gamma$. 

\item  There are at most countably many maximal tubes $\{M_j\}$.

\item  Let $\{S_k \}_{k\in\mbN_+}$ be a sequence of approximating tubes for $M$. Then it holds that
\begin{align*}
\lim_{k\rightarrow +\infty}\mcH^2(M\setminus S_k)=0.    
\end{align*}
\end{enumerate}
\end{lemma}
\begin{proof}
(\Rn{1}) Let $M_1=\bigcup_{k\in \mbN_+} S_{1,k}$ and $M_2=\bigcup_{k\in \mbN_+} S_{2,k}$ be two maximal tubes. Assume there exists $x\in M_1\cap M_2$. Then there exist regular tubes $S_{1,k_0}\subset M_1$ and $S_{2,k_0}\subset M_2$ with $x\in S_{1,k_0}\cap S_{2,k_0}$. Now for every $k\ge k_0$, $\Tilde{S}_k\vcentcolon=S_{1,k}\cup S_{2,k}$ is also a regular tube, and $\Tilde{S}_k\Subset \Tilde{S}_{k+1}$. Hence $\Tilde{M}\vcentcolon=\bigcup_{k\ge k_0}\Tilde{S}_k$ contains $M_1$ and $M_2$. By the maximality of $M_1$ and $M_2$, it must hold that $M_1=M_2=\Tilde{M}$.

(\Rn{2}) The uniqueness is an immediate consequence of part (\Rn{1}). Let $\Gamma$ be a non-contractible closed streamline. By Lemma \ref{prelim_lemma}-(\Rn{1}), there exists a regular tube $S_1$ containing $\Gamma$ (one can always use the same argument to find a larger regular tube containing $S_1$). So the collection of all regular tubes containing $\Gamma$, denoted by $\Lambda$, is nonempty. We claim that
\begin{align*}
M\vcentcolon=\bigcup_{S\in \Lambda} S
\end{align*}
is a maximal tube. First, the maximality of $M$ follows immediately from its definition. Next, we construct approximating tubes inductively.

The first regular tube $S_1$ has been constructed above. Assume that for some $k\in\mbN_+$, there are regular tubes
\[
\{S_1,S_2,\cdots,S_k\}\subset\Lambda \ \ {\rm satisfying}\  \ S_1\Subset S_2\cdots\Subset  S_k. 
\]
There exists a regular tube $\hat{S}\in \Lambda$ such that $S_k\Subset\hat{S}$. Let $\alpha=\sup_{S\in\Lambda}\mcH^2(S)$. Then there exists another regular tube $\Tilde{S}\in\Lambda$ such that 
\[
\alpha-\frac1k<\mcH^2(\Tilde{S})<\alpha.
\]
Now $S_{k+1}\vcentcolon=\hat{S}\cup \Tilde{S}$ is a regular tube belonging to $\Lambda$. Clearly, it holds that $S_k\Subset S_{k+1}$ and $\lim_{k\rightarrow+\infty}\mcH^2(S_k)=\alpha$.

We claim that $M=\bigcup_{k\in\mbN_+}S_k$. We only need to show $M\subset\bigcup_{k\in\mbN_+}S_k.$ If not, there exists $x\in M$ such that $x\notin S_k$ for any $k\in\mbN_+$. By the definition of $M$, there exists $S_x\in\Lambda$ containing $x$, and hence containing the streamline $\gamma(x)$ passing through $x$. Without loss of generality, we may assume that $\gamma(x)$ lies above $\Gamma$, whence it lies above every $S_k$. But then $S_x\cup S_k\in \Lambda$ satisfies that $\mcH^2(S_x\cup S_k)-\mcH^2(S_k)$ has a uniform positive lower bound, since $(S_x\cup S_k)\setminus S_k$ contains the regular tube bounded by $\gamma(x)$ and the upper boundary of $S_x$. This contradicts the fact that $\lim_{k\rightarrow+\infty}\mcH^2(S_k)=\alpha$, completing the proof of the claim.

(\Rn{3}) For every $k\in \mbN_+$, there are only finitely many maximal tubes with area greater than $\frac1k.$ We can sort the maximal tubes in descending order of area. Clearly, this procedure ensures that all maximal tubes are exhausted.

(\Rn{4}) Using the dominated convergence theorem yields
\begin{align*}
\lim_{k\rightarrow\infty}\mcH^2(M\setminus S_k)
=\iint_{M} 1\,dx-\lim_{k\rightarrow\infty}\iint_{M}1_{S_k}(x)\,dx=0.
\end{align*}
Hence the lemma is proved.
\end{proof}

\begin{lemma}\label{lemma_union_maxtubes}
If $\bv$ has no contractible closed streamlines, then there are at most countably many pairwise disjoint maximal tubes $\{M_j\}$ such that $u^{-1}(\msR(u))\subset \bigcup_{j}M_j$. Moreover,
\begin{align*}
\mcH^2\left(\bigcup_{j}M_j\setminus u^{-1}(\msR(u)) \right)=0.
\end{align*}
\end{lemma}
\begin{proof}
By Lemma \ref{lemma_maxtube}, the maximal tubes $\{M_j\}$ are at most countably many and pairwise disjoint.
For every $x\in u^{-1}(\msR(u)),$ by the assumption, the streamline $\gamma(x)$ passing through $x$ is non-contractible. Again, by Lemma \ref{lemma_maxtube}, there exists a maximal tube containing $\gamma(x)$. So $u^{-1}(\msR(u))\subset \bigcup_{j}M_j$. Since a maximal tube does not contain any critical point of $u$, we get from \eqref{decomposition_strip} that
\begin{align*}
\bigcup_{j}M_j\setminus u^{-1}(\msR(u))\subset u^{-1}(\msC(u))\cap \{|\nabla u|>0 \}.
\end{align*}
It follows from Lemma \ref{lemma_null} that the set on the right is $\mcH^2$-null. Hence the lemma follows.
\end{proof}

We are now ready to prove Theorem \ref{main_thm}-\ref{thm_B}.

\begin{proof}[Proof of Theorem \ref{main_thm}-\ref{thm_B}]
Assume that $\bv$ has no contractible closed streamlines. We will show that this assumption forces $\bv$ to be a shear flow. The decisive step is to prove the vanishing of the total curvature, namely
\begin{align}\label{trivial_tc}
\iint_{\mbT\times(-1,1)}(|\nabla\bv|^2-|\nabla|\bv||^2)\,dx=0.
\end{align}

First, by Lemma \ref{lemma_null}, $u^{-1}(\msC(u))\cap \{|\nabla u|>0 \}$ is an $\mcH^2$-null set. On the other hand, we have $|\nabla\bv|^2-|\nabla|\bv||^2=0$ $\mcH^2$-a.e. in $\{|\nabla u|=0\}$. Therefore, it follows from \eqref{decomposition_strip} and Lemma \ref{lemma_union_maxtubes} that
\begin{align}\label{tc_identity}
\iint_{\mbT\times(-1,1)}(|\nabla\bv|^2-|\nabla|\bv||^2)\,dx
=&\iint_{u^{-1}(\msR(u))}(|\nabla\bv|^2-|\nabla|\bv||^2)\,dx\nonumber\\
=&\sum_{j}\iint_{M_j}(|\nabla\bv|^2-|\nabla|\bv||^2)\,dx,
\end{align}
where $\{ M_j \}$ is the collection of maximal tubes. 

Let $M$ be any maximal tube, and $\{S_k\}_{k\in\mbN_+}$ be a sequence of approximating tubes for $M$. Since $|\bv|>0$ in $\overline{S_k}$ and $\overline{S_k}$ are foliated by non-contractible closed streamlines, there exists a function $\theta\in C^1(\overline{S_k})$  with $\bv/|\bv|=(\cos\theta,\sin\theta)$. In fact, $\theta(x)$ is exactly the flow angle at $x$. Here, $\theta\in C^1(\overline{S_k})$ is, by default, a periodic function in $x_1$; that is, $\theta(x)=\theta(x_1+2\pi,x_2)$ for all $x\in S_k$.
One then readily computes that
\begin{align}\label{grad_theta}
\nabla\theta=\frac{v_1\nabla v_2-v_2\nabla v_1}{|\bv|^2}\quad {\rm in}\quad S_k.
\end{align}
This together with \eqref{Euler_divergence_form} implies that
\begin{align*}
\Div(|\bv|^2\nabla\theta)=0 \quad {\rm in}\quad \mathcal{D}'(S_k),
\end{align*}
where $\mathcal{D}'(S_k)$ denotes the space of distributions in $S_k$.
Testing this equation by $\theta$ and noticing that
\begin{align*}
|\bv|^2|\nabla\theta|^2=|\nabla\bv|^2-|\nabla|\bv||^2 \quad {\rm in}\ S_k,
\end{align*}
we obtain
\begin{align}\label{elliptic_estimate}
\iint_{S_k}(|\nabla\bv|^2-|\nabla|\bv||^2)\,dx=\int_{\partial S_k}\theta |\bv|^2\partial_{\bn}\theta\,ds,
\end{align}
where $\partial_{\bn}$ denotes the outward normal derivative, and $s$ denotes the arc length parameter for the streamlines.

It follows directly from the Euler system \eqref{steady_Euler} that
\begin{align*}
v_1\nabla v_2-v_2\nabla v_1=\nabla^\perp P.
\end{align*}
Combining this with \eqref{grad_theta} yields
\begin{align}\label{normal_dtheta}
|\bv|^2\partial_{\bn}\theta=\bn\cdot\nabla^\perp P=-\bn^\perp\cdot\nabla P=-\partial_sP \quad {\rm on}\ \partial S_k,
\end{align}
where $\partial_s=\bn^\perp\cdot\nabla$ has been used.
Similarly,
\begin{align}\label{tangent_dtheta}
|\bv|^2\partial_{s}\theta=\partial_{\bn}P \quad {\rm on}\ \partial S_k .
\end{align}
Since $\partial S_k$ are streamlines, taking the inner product between $\bv$ and the momentum equation of \eqref{steady_Euler} gives the Bernoulli's law,
\begin{align}\label{Bernoulli}
\partial_sP=-\frac12\partial_s(|\bv|^2) \quad {\rm on}\ \partial S_k.
\end{align}

Next, the identity \eqref{elliptic_estimate}, together with \eqref{normal_dtheta}--\eqref{Bernoulli}, yields
\begin{align}\label{integral_on_regular_tube}
\iint_{S_k}(|\nabla\bv|^2-|\nabla|\bv||^2)\,dx
=&-\int_{\partial S_k}\theta\partial_s P\,ds
=\frac12\int_{\partial S_k}\theta\partial_s(|\bv|^2)\,ds\nonumber\\
=&-\frac12\int_{\partial S_k}|\bv|^2\partial_s\theta\,ds
=-\frac12\int_{\partial S_k}\partial_{\bn}P\,ds\nonumber\\
=& -\frac12\iint_{S_k}\Delta P\,dx.
\end{align}
We need to justify the regularity of $P$ in the last equality. In fact, it follows from the Euler system \eqref{steady_Euler} that
\begin{align}\label{Poisson_equation_P}
-\Delta P=\partial_{x_1}\bv\cdot\nabla v_1+\partial_{x_2}\bv\cdot\nabla v_2.
\end{align}
So $\Delta P\in C(\mbT\times[-1,1])$, and elliptic regularity implies $P\in W^{2,q}_{\textup{loc}}(\mbT\times(-1,1))$ for any $q\in[1,\infty)$.
Now sending $k$ to infinity in \eqref{integral_on_regular_tube} and recalling Lemma \ref{lemma_maxtube}, we arrive at
\begin{align*}
\iint_{M}(|\nabla \bv|^2-|\nabla|\bv||^2)\,dx
=-\frac12\iint_{M}\Delta P\,dx.
\end{align*}
Since this identity holds for every maximal tube $M$, we get from \eqref{tc_identity} and Lemma \ref{lemma_union_maxtubes} that
\begin{align}\label{tc_equals_P}
\iint_{\mbT\times(-1,1)}(|\nabla \bv|^2-|\nabla|\bv||^2)\,dx
=-\frac12\iint_{u^{-1}(\msR(u))}\Delta P\,dx.
\end{align}
Moreover, \eqref{Poisson_equation_P} implies that $\Delta P=0$ $\mcH^2$-a.e. in $\{|\nabla u|=0\}$. So by \eqref{decomposition_strip}, Lemma \ref{lemma_null} and \eqref{tc_equals_P}, we have
\begin{align*}
\iint_{\mbT\times(-1,1)}(|\nabla \bv|^2-|\nabla|\bv||^2)\,dx
=&-\frac12\iint_{\mbT\times(-1,1)}\Delta P\,dx\\
=&-\frac12\int_{\mbT}\partial_{x_2}P(x_1,1)\,dx_1+\frac12\int_{\mbT}\partial_{x_2}P(x_1,-1)\,dx_1.
\end{align*}
Furthermore, it follows from the Euler system and the slip boundary condition that
\[
\partial_{x_2}P(x_1,\pm 1)=-\bv\cdot\nabla v_2(x_1,\pm1)=0,\quad x_1\in\mbT.
\]
Consequently, the identity \eqref{trivial_tc} must hold.

Lastly, \eqref{trivial_tc} implies that the flow angle $\theta$ is constant in every maximal tube. This, together with the slip boundary condition,  forces $\bv$ to be a parallel shear flow.

Hence, the proof of Theorem \ref{main_thm}-\ref{thm_B} is completed.
\end{proof}


\medskip

\section{Selection principle in a bounded connected domain}\label{sec_selection_bdd}

This section is devoted to the proof of Theorem \ref{main_thm}-\ref{thm_A}. In fact, we can prove a slightly stronger result by considering the following forced Navier-Stokes system,
\begin{eqnarray}\label{NS_forced}
\left\{\begin{aligned}
&-\nu\Delta\bv^\nu+\bv^\nu\cdot\nabla\bv^\nu+\nabla P^\nu=\bbf^\nu,\\
&\Div\,\bv^\nu=0,
\end{aligned}\right.
\end{eqnarray}
where the external force satisfies $\bbf^\nu=o(\nu)$ as $\nu\rightarrow0^+$.

Clearly, Theorem \ref{main_thm}-\ref{thm_A} follows from the following proposition.
\begin{pro}\label{prop_A}
Suppose $\bv\in C^{2}(\Bar{\Omega})$ and $P\in C^{2}(\Bar{\Omega})$ solve \eqref{steady_Euler}-\eqref{slip boundary condition}, where $\Omega$ is of the form \eqref{defn_Omega}. 
Let $\bv^\nu\in C^2(\Omega)\cap C(\Bar{\Omega})$ and $P^\nu\in C^1(\Omega)$ be a sequence of solutions to \eqref{NS_forced}, where $\bbf^\nu\in C(\Omega)$ satisfies 
\begin{align}\label{smaller_force1}
\lim_{\nu\rightarrow 0^+} \frac{1}{\nu}\int_{\Omega'}|\bbf^\nu|\,dx=0,\quad {\rm for\ all}\ \Omega'\Subset\Omega.
\end{align}
If $\Omega$ is multiply connected, we further assume the zero-flux boundary condition
\begin{align}\label{zero_flux_bdd2}
\int_{\bGm_j}\bn\cdot\bv^\nu \,ds=0, \quad j=0,1,\cdots,J,
\end{align}
where $\bn$ is the outward unit normal to $\partial\Omega$.
If $\bv^\nu$ converges to $\bv$ in $L^\infty_{\textup{loc}}\cap H^1_{\textup{loc}}(\Omega)$ as $\nu$ tends to zero, then the vorticity $\omega=\partial_{x_1}v_2-\partial_{x_2}v_1$ must be constant.    
\end{pro}
\begin{proof}
Since $\bv$ is divergence-free and satisfies the slip boundary condition \eqref{slip boundary condition}, it admits a stream function $u\in C^3(\Bar{\Omega})$ such that $u=\bar{u}_j$ on $\bGm_j$ for $j=0,1,\cdots,J$, where $\bar{u}_0,\bar{u}_1,\cdots,\bar{u}_J$ are constants.
As before, $\msR(u)$ and $\msC(u)$ denote the set of regular values and the set of critical values of $u$, respectively.

{\it Step 1.} First, $\Omega$ can be decomposed into the disjoint union
\begin{align*}
\Omega=u^{-1}(\msR(u))\cup \{|\nabla u|=0 \} \cup (u^{-1}(\msC(u))\cap \{|\nabla u|>0 \}).
\end{align*}
We will keep the convention that the subsets on the right are relative to the open domain $\Omega$. 
It follows from Lemma \ref{lemma_null} that $u^{-1}(\msC(u))\cap \{|\nabla u|>0 \}$ is an $\mcH^2$-null subset of $\Omega$.
For $u\in C^3(\Bar{\Omega})$, it holds that
\begin{align*}
\nabla\omega=\Delta \nabla u=0,\quad \mcH^2\textup{-a.e. in } \ \{|\nabla u|=0 \}.
\end{align*}
Consequently,
\begin{align}\label{gra_omega1}
\nabla\omega=0,\quad \mcH^2\text{-a.e. in }\ \Omega\setminus u^{-1}(\msR(u)).
\end{align}
Therefore, it remains to show that
\begin{align}\label{gra_omega2}
\nabla\omega=0,\quad \text{in} \ u^{-1}(\msR(u)).
\end{align}

{\it Step 2.} 
For every $r\in\msR(u)$, $u^{-1}(r)$ is a $C^3$ closed one-dimensional manifold in $\Omega$ (we adopt the convention that, if $r\in\{\bar{u}_{0},\bar{u}_{1},\cdots,\bar{u}_{J} \},$ $u^{-1}(r)$ does not contain the boundary component).
Then by the classification of closed one-dimensional manifolds, $u^{-1}(r)$ consists of finitely many $C^3$ Jordan curves in $\Omega.$ Let $\Gamma_r$ be any component of $u^{-1}(r)$. 
Let us denote by $\bn$ the outward unit normal to $\Gamma_r$, and assume that $\bn\cdot\nabla u=|\nabla u|$ (where $|\nabla u|>0$). The opposite case, $\bn\cdot\nabla u=-|\nabla u|$, can be handled similarly.

Now let $u^\nu$ be a stream function of $\bv^\nu$, whose existence is guaranteed by the incompressibility of $\bv^\nu$ and the boundary condition \eqref{zero_flux_bdd2}. Fix $z\in \Omega$ and specify $u^\nu(z)=u(z).$ Then the assumption of Proposition \ref{prop_A} and Lemma \ref{prelim_u_convergence} imply that 
\begin{align}\label{uep_converges_u}
u^\nu \rightarrow u\quad  {\rm in} \quad C_{\text{loc}}^1(\Omega)\cap H^2_{\text{loc}}(\Omega) \quad {\rm as} \quad \nu\rightarrow 0^+.
\end{align}

Recall Lemma \ref{prelim_lemma} and the notations introduced therein. For $t\in [-\delta_0,\delta_0]$, define 
\begin{align*}
X(t)=\int_{\Gamma_{r+t}}\frac{\nabla u}{|\nabla u|}\cdot\nabla\omega\,ds,\quad 
Y(t)=\int_{-\delta_0}^{t}X(t')\,dt', \quad\text{and}\quad 
Z(t)=\int_{-\delta_0}^{t}Y(t')\,dt'.
\end{align*}
Clearly, $X$ is continuous with respect to $t$,
so $Z''(t)=X(t)$ for every $t\in [-\delta_0,\delta_0]$.
By the coarea formula and integration by parts, one has
\begin{align*}
Y(t)=& \int_{-\delta_0}^{t} \int_{\Gamma_{r+t'}}\frac{\nabla u}{|\nabla u|}\cdot\nabla\omega\,ds \,dt'\\
=&\iint_{A_{-\delta_0,t}}\nabla u\cdot\nabla\omega\,dx\\
=&-\iint_{A_{-\delta_0,t}}\omega^2\,dx
+\int_{\Gamma_{r+t}}\omega|\nabla u|\,ds
-\int_{\Gamma_{r-\delta_0}}\omega|\nabla u|\,ds.
\end{align*}
Applying the coarea formula again yields
\begin{align*}
Z(t)=&-\int_{-\delta_0}^{t} \iint_{A_{-\delta_0,t'}}\omega^2\,dx\, dt'
+\iint_{A_{-\delta_0,t}}\omega|\nabla u|^2\,dx
-(t+\delta_0)\int_{\Gamma_{r-\delta_0}}\omega|\nabla u|\,ds.
\end{align*}
Furthermore, it holds that
\begin{align*}
\int_{-\delta_0}^{t} \iint_{A_{-\delta_0,t'}}\omega^2\,dx\, dt'
=&\int_{-\delta_0}^{t} \int_{-\delta_0}^{t'}\int_{\Gamma_{r+t''}}\frac{\omega^2}{|\nabla u|} \,ds \, dt'' \, dt'\\
=&\int_{-\delta_0}^{t} \int_{t''}^{t} \int_{\Gamma_{r+t''}}\frac{\omega^2}{|\nabla u|} \,ds \, dt' \, dt''\\
=&\int_{-\delta_0}^{t} \int_{\Gamma_{r+t''}} (r+t-u) \frac{\omega^2}{|\nabla u|} \,ds \, dt''\\
=&\int_{A_{-\delta_0,t}}(r+t-u)\omega^2\,dx.
\end{align*}
Therefore, for every $t\in [-\delta_0,\delta_0]$, we arrive at 
\begin{align*}
&Z(t)+(t+\delta_0) \int_{\Gamma_{r-\delta_0}}\omega|\nabla u|\,ds
=\iint_{A_{-\delta_0,t}}\left( \omega|\nabla u|^2-(r+t-u)\omega^2 \right)\,dx.
\end{align*}

Similarly, for $t\in [-\delta_0,\delta_0]$, define
\begin{align*}
X^\nu(t)=\int_{\Gamma_{r+t}^\nu}\frac{\nabla u^\nu}{|\nabla u^\nu|}\cdot\nabla\omega^\nu \,ds
\quad\text{and}\quad 
Z^\nu(t)=\int_{-\delta_0}^{t}\int_{-\delta_0}^{t'}X^\nu(t'') \,d t''\,d t'.   
\end{align*}
The same computations yield that for $t\in [-\delta_0,\delta_0]$,
\begin{align*}
&Z^\nu(t)+(t+\delta_0) \int_{\Gamma_{r-\delta_0}^\nu}\omega^\nu|\nabla u^\nu|\,ds
=\int_{A^\nu_{-\delta_0,t}}\left( \omega^\nu|\nabla u^\nu|^2-(r+t-u^\nu)(\omega^\nu)^2 \right)\,dx.
\end{align*}

For every $t\in [-\delta_0,\delta_0]$, as $\nu\rightarrow0^+,$ it follows from \eqref{uep_converges_u} that
\[
Q^\nu\vcentcolon=\omega^\nu|\nabla u^\nu|^2-(r+t-u^\nu)(\omega^\nu)^2\to Q\vcentcolon=\omega|\nabla u|^2-(r+t-u)\omega^2 \quad \text{in} \ L^1(A_{-2\delta_0,2\delta_0}).
\]
This, together with Lemma \ref{prelim_lemma}-(\Rn{3}) and
\[
\left| \int_{A^\nu_{-\delta_0,t}}Q^\nu\,dx-\int_{A_{-\delta_0,t}}Q \,dx\right|
\le \int_{A_{-2\delta_0,2\delta_0}}|Q^\nu-Q| \,dx
+\int_{A_{-\delta_0,t}\triangle A^\nu_{-\delta_0,t}}|Q|\,dx,
\]
yields
\begin{align*}
\lim_{\nu\rightarrow 0^+}
\int_{A^\nu_{-\delta_0,t}}Q^\nu\,dx
=\int_{A_{-\delta_0,t}}Q \,dx,\quad t\in [-\delta_0,\delta_0].
\end{align*}
It then follows that for $t\in [-\delta_0,\delta_0]$,
\begin{align}\label{limit_trick}
&\lim_{\nu\rightarrow 0^+}
\left( Z^\nu(t)+(t+\delta_0)\int_{\Gamma^{\nu}_{r-\delta_0}}\omega^\nu|\nabla u^\nu|\,ds \right)
=
Z(t)+(t+\delta_0)\int_{\Gamma_{r-\delta_0}}\omega|\nabla u|\,ds.
\end{align}
This formula will also be used in the next section for the proof of Theorem \ref{main_thm}-\ref{thm_C}.

{\it Step 3.}
Taking the dot product between $\frac{\bv^\nu}{|\bv^\nu|}$ and the momentum equation of \eqref{NS_forced} yields
\begin{align}\label{total_head_pressure}
\frac{\bv^\nu}{|\bv^\nu|}\cdot\nabla \left(\frac12|\bv^\nu|^2+P^\nu\right)
=\nu \frac{\bv^\nu}{|\bv^\nu|}\cdot\nabla^\perp\omega^\nu+\bbf^\nu\cdot\frac{\bv^\nu}{|\bv^\nu|},   
\end{align}
where $\omega^\nu=\partial_{x_1}v_2^\nu-\partial_{x_2}v_1^\nu$ is the vorticity of $\bv^\nu$, and the identity $\Delta\bv^\nu=\nabla^\perp\omega^\nu$ has been used. Next, notice that the integral on the left-hand side of \eqref{total_head_pressure} along $\Gamma_{r+t}^\nu$ ($t\in [-\delta_0,\delta_0]$) vanishes.
Consequently, for $t\in [-\delta_0,\delta_0]$, it holds that
\begin{align*}
X^\nu(t)=\int_{\Gamma_{r+t}^\nu}\frac{\bv^\nu}{|\bv^\nu|}\cdot\nabla^\perp\omega^\nu \,ds
=
-\frac{1}{\nu}\int_{\Gamma_{r+t}^\nu} \bbf^\nu\cdot \frac{\bv^\nu}{|\bv^\nu|} \,ds.   
\end{align*}
By the definition of $Z^\nu(t)$, one has
\begin{align*}
Z^\nu(t)=&-\frac{1}{\nu}
\int_{-\delta_0}^{t}\int_{-\delta_0}^{t'}\int_{\Gamma_{r+t''}^\nu} \bbf^\nu\cdot \frac{\bv^\nu}{|\bv^\nu|} \,ds \,d t''\,d t'
= -\frac{1}{\nu}\iint_{A^\nu_{-\delta_0,t}}(r+t-u^\nu)\bbf^\nu\cdot\bv^\nu\,dx.
\end{align*}
This, together with assumption \eqref{smaller_force1}, implies that
\[
\lim_{\nu\rightarrow0^+} Z^\nu(t)=0,\quad t\in [-\delta_0,\delta_0].
\]
It then follows from \eqref{limit_trick} that the limit
\[
l_0\vcentcolon=\lim_{\nu\rightarrow 0^+}
\int_{\Gamma^{\nu}_{r-\delta_0}}\omega^\nu|\nabla u^\nu|\,ds
\]
exists, and that
\begin{align*}
Z(t)=l_0(t+\delta_0)-(t+\delta_0)\int_{\Gamma_{r-\delta_0}}\omega|\nabla u|\,ds,\quad t\in [-\delta_0,\delta_0].
\end{align*}
Hence one has $X(t)=Z''(t)=0$ for every $t\in[-\delta_0,\delta_0]$.
In particular, taking $t=0$ gives
\begin{align}\label{int_gra_omega}
\int_{\Gamma_r}\frac{\nabla u}{|\nabla u|}\cdot\nabla\omega\,ds=0.
\end{align}

Since $A_{-\delta_0,\delta_0}$ is foliated by streamlines, then there exists a $C^1$ function $f$ such that $\omega=f(u)$ in $A_{-\delta_0,\delta_0}$ (see, e.g., \cite{HN_CPAM_2017,HN_JEMS_2023}). This fact along with \eqref{int_gra_omega} gives that $f'(r)\int_{\Gamma_r}|\nabla u|\,ds=0.$ Recall that $|\nabla u|>0$ on $\Gamma_r$. Hence we have $f'(u)=f'(r)=0$ on $\Gamma_r$, which further implies $\nabla\omega=0$ on $\Gamma_r$. Given that both $r\in\msR(u)$ and $\Gamma_r\subset u^{-1}(r)$ are arbitrary, \eqref{gra_omega2} must hold.

Finally, it follows from \eqref{gra_omega1} and \eqref{gra_omega2} that $\nabla\omega=0$ $\mcH^2$-a.e. in $\Omega$, hence everywhere in $\Omega$ by the continuity of $\nabla\omega$. Thus $\omega$ is constant in $\Omega$, and the proof of Proposition \ref{prop_A} is completed.
\end{proof}

\begin{remark}
The requirement for second-order derivative regularity of $\bv$ is essential in the above argument. It would be very interesting to characterize vanishing viscosity limits for flows with lower regularity.
\end{remark}


\medskip

\section{A selection principle for flows in periodic strips}\label{sec_selection_strip}

In this section, we apply Theorem \ref{main_thm}-\ref{thm_A} and Theorem \ref{main_thm}-\ref{thm_B} to prove Theorem \ref{main_thm}-\ref{thm_C}.

In fact, Theorem \ref{main_thm}-\ref{thm_C} is a direct consequence of the following proposition.

\begin{pro}\label{prop_B}
Suppose that $\bv\in C^{2}(\mbT\times[-1,1])$ and $P\in C^{2}(\mbR\times[-1,1])$ solve \eqref{steady_Euler} supplemented with $v_2(\cdot,\pm1)=0$. 
Let $\bv^\nu\in C^{2}(\mbT\times(-1,1))\cap C(\mbT\times[-1,1])$ and $P^\nu\in C^1(\mbR\times(-1,1))$ be a sequence of solutions to \eqref{NS_forced}, where $\bbf^\nu\in C(\mbT\times(-1,1))$ satisfies 
\begin{align}\label{smaller_force2}
\lim_{\nu\rightarrow 0^+} \frac{1}{\nu}\int_{\mbT\times(-1+h,1-h)}|\bbf^\nu|\,dx=0,\quad {\rm for\ all}\ h\in(0,1).
\end{align}
Assume $\bv^\nu=(v_1^\nu,v_2^\nu)$ satisfies the zero-flux boundary condition
\begin{align}\label{zero_flux_strip2}
\int_{\mbT}v_2^\nu(x_1,-1)\,dx_1=0.  
\end{align}
If $\bv^\nu$ converges to $\bv$ in $L^\infty_{\textup{loc}}\cap H^1_{\textup{loc}}(\mbT\times(-1,1))$ as $\nu$ tends to zero, then there exist constants $c_0,c_1,$ and $c_2$ such that $\bv=(c_0+c_1 x_2+c_2 x_2^2,0)$.  
\end{pro}

We adopt the notation from Section \ref{sec_rigidity}. Let $u\in C^3(\mbT\times[-1,1])$ be the stream function of $\bv$ with $u(\cdot,1)=0$ and $u(\cdot,-1)= c$ on $\mbT$. 
Let $\msR(u)$ and $\msC(u)$ denote the set of regular values and the set of critical values of $u$, respectively. Remember that $\mbT\times(-1,1)$ can be decomposed into the disjoint union
\begin{align}\label{decomposition_strip_again}
\mbT\times(-1,1)=u^{-1}(\msR(u))\cup \{|\nabla u|=0 \} \cup (u^{-1}(\msC(u))\cap \{|\nabla u|>0 \}).
\end{align}
For every $r\in\msR(u)$, $u^{-1}(r)$ consists of finitely many $C^3$ contractible or non-contractible closed streamlines. 

The proof of Proposition \ref{prop_B} falls into three cases, according to whether $u^{-1}(\msR(u))$ contains only contractible, only non-contractible, or both types of closed streamlines. We shall clarify that $u^{-1}(\msR(u))$ contains a contractible (resp. non-contractible) closed streamline if and only if $\mbT\times(-1,1)$ does. To see this, suppose $\Gamma_r\subset u^{-1}(r)$ is a contractible (resp. non-contractible) closed streamline in $\mbT\times(-1,1)$. By Lemma \ref{prelim_lemma}-(\Rn{1}), we know that there exists a tubular neighborhood of $\Gamma_r$ that is foliated by contractible (resp. non-contractible) closed streamlines $\Gamma_s\subset u^{-1}(s)$, $r-\delta< s<r+\delta$. Since $\msR(u)$ has full measure in $u(\mbT\times[-1,1])$ by the Morse-Sard theorem, it must contain some $s_0\in(r-\delta,r+\delta)$. Then $\Gamma_{s_0}$ is a contractible (resp. non-contractible) closed streamline in $u^{-1}(\msR(u))$.

We are now in a position to prove Proposition \ref{prop_B}.

\begin{proof}[Proof of Proposition \ref{prop_B}]
The proof is divided into six steps.

{\it Step 1.} 
For any non-contractible closed streamline $\Gamma$ of $\bv$ (should one exist), with upward-pointing unit normal $\bn$, we claim that there exists a constant $b$ independent of $\Gamma$ such that
\begin{align}\label{const_int}
\int_{\Gamma}\bn\cdot\nabla\omega\,ds=b.
\end{align}

The proof of \eqref{const_int} is similar to that of \eqref{int_gra_omega}, so we only provide a sketch.
First, it follows from the Navier-Stokes equations \eqref{NS_forced} and the periodicity of $\bv^\nu$ and $\bbf^\nu$ that $\nabla P^\nu$ is also periodic in $x_1$. Hence, there exists a constant $a^\nu$ depending only on $\nu$ such that
\begin{align*}
P^\nu(x_1,x_2)-P^\nu(x_1+2\pi,x_2)=a^\nu \quad {\rm for\ all}\ x\in\mbR\times(-1,1).
\end{align*}
By \eqref{zero_flux_strip2}, $\bv^\nu$ admits a periodic stream function. So the results established in Lemma \ref{prelim_u_convergence} and Lemma \ref{prelim_lemma} are applicable, and if $\bv$ possesses a non-contractible closed streamline, then so does $\bv^\nu$ for all sufficiently small $\nu$. 
Let $\Gamma^\nu$ be any non-contractible closed streamline of $\bv^\nu$. Denote by $\bn^\nu$ the upward-pointing unit normal of $\Gamma^\nu$. We can rewrite \eqref{total_head_pressure} as
\begin{align*}
(\bn^\nu)^\perp\cdot\nabla \left(\frac12|\bv^\nu|^2+P^\nu\right)
=\nu \bn^\nu\cdot\nabla \omega^\nu+(\bn^\nu)^\perp\cdot\bbf^\nu.   
\end{align*}
The integral of the left-hand side of the above equation along $\Gamma^\nu$ equals $a_\nu$. 
Therefore, one has
\begin{align}\label{integral_constraint_epsl2}
\int_{\Gamma^\nu}\bn^\nu\cdot\nabla\omega^\nu\,ds=b^\nu-\frac{1}{\nu}\int_{\Gamma^\nu}(\bn^\nu)^\perp\cdot\bbf^\nu\,ds,
\end{align}
where $b^\nu=\frac{a^\nu}{\nu}$.
Let $\Gamma_r\subset u^{-1}(r)$ be any non-contractible closed streamline of $\bv$. We refer the reader to the proof of Proposition \ref{prop_A} for the definition of the functions
\begin{align*}
X(t)=\int_{\Gamma_{r+t}}\frac{\nabla u}{|\nabla u|}\cdot\nabla\omega\,ds,\quad 
Z(t)=\int_{-\delta_0}^{t}\int_{-\delta_0}^{t'}X(t'') \,d t''\,d t',
\quad t\in [-\delta_0,\delta_0],
\end{align*}
and
\begin{align*}
X^\nu(t)=\int_{\Gamma_{r+t}^\nu}\frac{\nabla u^\nu}{|\nabla u^\nu|}\cdot\nabla\omega^\nu \,ds,\quad  
Z^\nu(t)=\int_{-\delta_0}^{t}\int_{-\delta_0}^{t'}X^\nu(t'') \,d t''\,d t',
\quad t\in [-\delta_0,\delta_0].   
\end{align*}
Let us assume $\bn=\frac{\nabla u}{|\nabla u|}$, and hence $\bn^\nu=\frac{\nabla u^\nu}{|\nabla u^\nu|}$ (the case $\bn=-\frac{\nabla u}{|\nabla u|}$ is handled similarly). 
It follows from \eqref{integral_constraint_epsl2} that
\begin{align*}
Z^\nu(t)=\frac12 b^\nu (t+\delta_0)^2
-\frac{1}{\nu}\iint_{A^\nu_{-\delta_0,t}}(r+t-u^\nu)\bbf^\nu\cdot\bv^\nu\,dx.    
\end{align*}
Now recalling \eqref{smaller_force2} and \eqref{limit_trick}, for $t\in [-\delta_0,\delta_0]$, one has
\begin{align*}
&\lim_{\nu\rightarrow 0}
\left( \frac12 b^\nu (t+\delta_0)^2+(t+\delta_0)\int_{\Gamma^{\nu}_{r-\delta_0}}\omega^\nu|\nabla u^\nu|\,ds \right)
=
Z(t)+(t+\delta_0)\int_{\Gamma_{r-\delta_0}}\omega|\nabla u|\,ds.
\end{align*}
The pointwise convergence of the sequence of quadratic functions in $t+\delta_0$ (on the left-hand side of the above identity) implies the convergence of each of its coefficients. Namely, both
\[
b\vcentcolon=\lim_{\nu \rightarrow 0^+}b^\nu
\quad
\text{and}\quad 
l_0\vcentcolon=\lim_{\nu\rightarrow 0^+}\int_{\Gamma^{\nu}_{r-\delta_0}}\omega^\nu|\nabla u^\nu|\,ds
\]
exist. 
Since $b^\nu=\frac{a^\nu}{\nu}$ depends only on $\nu$, its limit $b$ must be independent of the choice of the non-contractible closed streamline $\Gamma_r$.
We proceed to the conclusion that
\begin{align*}
Z(t)=\frac12 b(t+\delta_0)^2+l_0(t+\delta_0)-(t+\delta_0)\int_{\Gamma_{r-\delta_0}}\omega|\nabla u|\,ds
,\quad \text{for\ any}\ t\in [-\delta_0,\delta_0].
\end{align*}
It immediately follows that $X(t)=Z''(t)=b$ for all $t\in [-\delta_0,\delta_0]$, thereby proving \eqref{const_int}.

{\it Step 2.} 
If $u^{-1}(\msR(u))$ does not contain any contractible closed streamline, it follows from Theorem \ref{main_thm}-\ref{thm_B} that $\bv$ must be a shear flow, say, $\bv=(g(x_2),0)$ for some $C^2$ function $g$. Then by \eqref{const_int}, we have 
\[
g''=-\frac{b}{2\pi}, \quad \text{in}\ \, (-1,1)\setminus\{g=0 \}.
\]
On the other hand, $g''=0$ $\mcH^1$-a.e. in $\{g=0 \}$. Since $g''$ is continuous in $(-1,1)$, it must hold that either $g''=-\frac{b}{2\pi}$ or $g''=0$ in the whole interval $(-1,1)$. So $\bv=(c_0+c_1 x_2+c_2 x_2^2,0)$ for some constants $c_0,c_1,$ and $c_2$.

{\it Step 3.}
If $u^{-1}(\msR(u))$ consists solely of contractible closed streamlines, Proposition \ref{prop_A} would imply that $\omega$ is constant. However, the only Euler flow with constant vorticity in $\mbT\times[-1,1]$ is a shear flow of the form $(c_0+c_1x_2,0)$, which leads to a contradiction.

{\it Step 4.} 
Assume that $u^{-1}(\msR(u))$ contains both contractible and non-contractible closed streamlines. In the remaining steps, we show that this last possibility leads to a contradiction as well.

Let $\cGm_1$ be a contractible closed streamline of $\bv$. By Proposition \ref{prop_A}, the vorticity $\omega$ is constant in $\text{Int}(\cGm_1)$. This constant is clearly nonzero; otherwise, $u$ would be constant in $\text{Int}(\cGm_1)$, contradicting $|\nabla u|>0$ on $\cGm_1$.
Without loss of generality, we may assume that
\begin{align*}
\omega=\Delta u=-1\quad {\rm in}\ \text{Int}(\cGm_1).
\end{align*}
Let $\Upsilon$ denote the collection of $\cGm_1$ and all contractible closed streamlines that contain $\cGm_1$ in their interior.
Define
\[
D=\bigcup_{\cGm\in\Upsilon}\text{Int}(\cGm).
\]
Applying Proposition \ref{prop_A} again, it is clear that
\begin{align}\label{torsion}
\omega=\Delta u=-1\quad {\rm in}\ \Bar{D}.    
\end{align}
Now we establish the following lemma in the spirit of Lemma \ref{lemma_maxtube}.

\begin{lemma}\label{lemma_maxeddy}
There exists a sequence $\{\cGm_k\}_{k\in\mbN_+}\subset\Upsilon$ such that $\cGm_k\subset\textup{Int}(\cGm_{k+1})$ for every $k\in\mbN_+$, and that $D=\bigcup_{k\in\mbN_+}\textup{Int}(\cGm_k)$. 
\end{lemma}
\begin{proof}
We construct $\{\cGm_k\}_{k\in\mbN_+}\subset\Upsilon$ inductively.
Assume that we have constructed
\[
\cGm_1,\cGm_2,\cdots,\cGm_k
\]
for some $k\in\mbN_+$. First, there exists $\Gamma\in\Upsilon$ with $\cGm_k\subset\textup{Int}(\Gamma)$.
Define
\[
\beta=\sup_{\cGm\in\Upsilon}\mcH^2(\text{Int}(\cGm)).
\]
Then there exists $\tilde{\Gamma}\in\Upsilon$ such that 
\[
\beta-\frac1k<\mcH^2(\text{Int}(\Tilde{\Gamma}))<\beta.
\]
Clearly, it holds that either $\Tilde{\Gamma}\subset\text{Int}(\Gamma)$ or $\Gamma\subset\text{Int}(\Tilde{\Gamma})$. Now define $\cGm_{k+1}=\Gamma$ if $\Tilde{\Gamma}\subset\text{Int}(\Gamma)$, or $\cGm_{k+1}=\Tilde{\Gamma}$ if $\Gamma\subset\text{Int}(\Tilde{\Gamma})$. It is clear that 
\[
\cGm_{k+1}\in\Upsilon, \quad \cGm_k\subset\textup{Int}(\cGm_{k+1}),\quad \text{and}\quad  \lim_{k\rightarrow\infty}\mcH^2(\text{Int}(\cGm_{k}))=\beta.
\]
We shall prove $D=\bigcup_{k\in\mbN_+}\textup{Int}(\cGm_k)$ by contradiction. Assume there exists $x\in D$ such that $x\notin \textup{Int}(\cGm_k)$ for any $k\in\mbN_+$. By the definition of $D$, there exists $\Hat{\Gamma}\in\Upsilon$ such that $x\in \text{Int}(\Hat{\Gamma}).$ Then it holds that $\cGm_k\subset \text{Int}(\Hat{\Gamma})$ for every $k\in\mbN_+$. But this leads to a contradiction that 
\[
\lim_{k\rightarrow\infty}\mcH^2(\text{Int}(\cGm_{k}))\le \mcH^2(\text{Int}(\Hat{\Gamma})) <\beta.
\]
Hence $D\subset\bigcup_{k\in\mbN_+}\textup{Int}(\cGm_k)$. This finishes the proof of Lemma \ref{lemma_maxeddy}.
\end{proof}

{\it Step 5.} 
For every $k\in\mbN_+$, $u$ is constant on $\cGm_k$, and we denote $c_k=u|_{\cGm_k}$. Applying the strong maximum principle to \eqref{torsion}, we infer that $c_k>c_{k+1}$ for every $k\in\mbN_+$. Thus, the limit $\lim_{k\rightarrow\infty}c_k$ exists, and we may assume without loss of generality that its value is zero. We conclude that
\begin{eqnarray}\label{positive_torsion}
\left\{\begin{aligned}
&\Delta u=-1, & {\rm in}\ \Bar{D},\\
&u>0, & {\rm in}\ D,\\
&u=0, & {\rm on}\ \partial D.
\end{aligned}\right.
\end{eqnarray}

Since $D$ is an open set, it contains a closed disk. Translate this closed disk in the horizontal direction until it first contacts the boundary $\partial D$ at some point $z$. Although $\partial D$ may contain boundary points of the strip, this procedure ensures that $z$ lies in the interior of the strip. Now $D$ satisfies the interior ball condition at $z\in\partial D$. We can apply Hopf's Lemma to the problem \eqref{positive_torsion}, and obtain $|\nabla u(z)|>0$. 
Denote $\bn(z)=\nabla u(z)/|\nabla u(z)|$.
There exists a disk $\overline{B_{\delta_0}(z)}$ (with $\delta_0>0$) lying entirely within the strip $\mbT\times(-1,1)$, on which $\bn(z)\cdot\nabla u>0$ and
\begin{align}\label{subsolution}
\omega=\Delta u\le -\frac12.
\end{align}
By the implicit function theorem, there exists $\delta_1>0$ such that $u^{-1}([-\delta_1,\delta_1])\cap B_{\delta_0}(z)$ are foliated by $\{\gamma_s \}_{s\in[-\delta_1,\delta_1]}$, where each $\gamma_s \vcentcolon=u^{-1}(s)\cap B_{\delta_0}(z)$ is a graph over an interval in the direction perpendicular to $\bn(z)$. As a consequence, there exists a $C^1$ function $f$ such that
\begin{align}\label{slee_in_tube}
\omega=\Delta u=f(u)\quad {\rm in}\  S_{-\delta_1,\delta_1},
\end{align}
where we define $S_{s_1,s_2}:=\bigcup_{s\in(s_1,s_2)}\gamma_s$.
Additionally, it is clear to see that $\gamma_0\subset\partial D$, $S_{0,\delta_1}\subset D$, and $S_{-\delta_1,0}\subset\mbT\times(-1,1)\setminus\Bar{D}$. We then differentiate \eqref{slee_in_tube} to get
\[
0=|\nabla\omega(z)|=|f'(u(z))||\nabla u(z)|.
\]
This implies
\begin{align}\label{fprime_zero}
f'(0)=f'(u(z))=0.    
\end{align}

{\it Step 6.} 
We show that the constancy of $\omega$ in $\bar{D}$ propagates across $\gamma_0$ to the entire strip, contradicting the presence of contractible closed streamlines.

For this, we first claim that for every $s\in(-\delta_1,0)$, $\gamma_{s}$ cannot be extended to a contractible closed streamline $\cGm_{s}\subset u^{-1}(s)$. Suppose, by contradiction, that such a streamline $\cGm_{s}$ exists. By the maximality of $D$, $\Bar{D}$ lies in the exterior of $\cGm_s$. So $S_{s,\delta_1}$ and $S_{-\delta_1,s}$ lie in the exterior and interior of $\cGm_s$, respectively. Then $\bnu\vcentcolon=\nabla u/|\nabla u|$ is the outward unit normal of $\cGm_s$ on the segment $\gamma_s$. However, it follows from Proposition \ref{prop_A} that $\omega$ is constant within $\text{Int}(\cGm_{s})$. So \eqref{subsolution} also holds in $\text{Int}(\cGm_{s})$. Applying Hopf's lemma yields $\bnu\cdot\nabla u<0$ on $\gamma_s$, contradicting the facts that $\bnu=\nabla u/|\nabla u|$ and $|\nabla u|>0$ on $\gamma_s$. This proves the claim.

\begin{figure}[htbp]
\centering
\begin{tikzpicture}[scale=1.6]
    \def\deltaZero{0.6}  
    \def\deltaSmall{0.15}
    \def\pointSize{1.2pt}
    
    \draw[thick, fill=gray!15] 
    (1.3,0.3) .. controls (1.8,2.0) and (2.8,2.2) .. 
    (4.0,1.8) .. controls (4.8,1.533) and (4.5,0.4) .. 
    (3.0,0.1) .. controls (1.9,-0.06) and (1.20,-0.04) .. 
    cycle;
    
    \node at (3.0,1.0) {$D$};
    
    \coordinate (z) at (2,1.57);
    
    \fill[black] (z) circle (\pointSize);
    \node[black, right] at (z) {$z$};
    
    \draw[dashed] (z) circle (\deltaZero);
    \node[above] at ($(z) + (0.2,-0.45)$) {$B_{\delta_0}(z)$};

    \draw[blue!15, line width=6pt] 
        (1,0.9) .. controls (1.5,1.6) and (3,3.1) .. 
        (4.5,1.6) .. controls (4.5,1.6) and (5,1.1) .. (6,0.9);
    \node[blue, above] at (5,1.8) {$\nGm_{-\delta}\subset S(\nGm_{-\delta})$};
    
    \draw[blue, thick] 
        (1,0.9) .. controls (1.5,1.6) and (3,3.1) .. 
        (4.5,1.6) .. controls (4.5,1.6) and (5,1.1) .. (6,0.9);
    
    \begin{scope}
        \path[clip] (z) circle (\deltaZero);
        \draw[red, line width=2pt] 
            ($(z) + (-0.9,-0.53)$) .. controls ($(z) + (-0.3,0.23)$) and ($(z) + (0.25,0.42)$) .. 
            ($(z) + (0.5,0.57)$);
    \end{scope}
    
    \node[red, below] at ($(z) + (-0.32,0.54)$) {$\gamma_{-\delta}$};

\end{tikzpicture}
\caption{Illustration for the proof of Proposition \ref{prop_B}.}
\label{illustration_proof}
\end{figure}

For every $\delta_2\in (0,\delta_1)$, it follows from \eqref{decomposition_strip_again} and Lemma \ref{lemma_null} that $S_{-\delta_2,0}\setminus (u^{-1}(\msC(u))\cap \{|\nabla u|>0 \})$ is a subset of $u^{-1}(\msR(u))$ with positive $\mcH^2$-measure. Consequently, for every $\delta_2\in (0,\delta_1)$, there exists $\delta\in(0,\delta_2)$ such that $-\delta\in\msR(u)$. By the above claim, we see that the component of $u^{-1}(-\delta)$ that contains $\gamma_{-\delta}$ is a non-contractible closed streamline. We therefore denote this component by $\nGm_{-\delta}$. 
One can further find a regular tube (see Definition \ref{defn_regular_tube}) $S(\nGm_{-\delta})$ containing $\nGm_{-\delta}$ (see Figure \ref{illustration_proof}). Then the equation \eqref{slee_in_tube} also holds in $S(\nGm_{-\delta})$.
This, along with \eqref{const_int}, implies that
\begin{align}\label{fprime_delta}
|f'(-\delta)| \int_{\nGm_{-\delta}}|\nabla u| \,ds=|b|.
\end{align}

For any non-contractible closed streamline $\nGm$, let $\Tilde{S}$ be the region bounded by $\nGm$ and $\mbT\times\{ -1\}$. Then one has
\begin{align*}
\iint_{\Tilde{S}}\Delta u\,dx
=\int_{\nGm}\bn\cdot\nabla u\,ds-\int_{\mbT}\partial_{x_2}u(x_1,-1)\,dx_1.
\end{align*}
Since either $\bn\cdot\nabla u\equiv |\nabla u|$ or $\bn\cdot\nabla u\equiv -|\nabla u|$ on $\nGm$, it follows that
\begin{align*}
\int_{\nGm}|\nabla u|\,ds\le \iint_{\mbT\times(-1,1)}|\Delta u|\,dx+\int_{\mbT}|\partial_{x_2}u(x_1,-1)|\,dx_1 <\infty.
\end{align*}
Combining this with \eqref{fprime_zero} and then taking the limit as $\delta\rightarrow 0$ in \eqref{fprime_delta} yields $b=0$.

Finally, since $b=0$ in \eqref{const_int}, the same reasoning as in the proof of Proposition \ref{prop_A} shows that $\omega$ is constant in $\mbT\times(-1,1)$. This contradicts the presence of contractible closed streamlines, thus completing the proof of Proposition \ref{prop_B}.
\end{proof}


\medskip 

{\bf Acknowledgement.}
The research of Gui is supported by NSFC Key Program (Grant No. 12531010), University of Macau research grants CPG2024-00016-FST, CPG2025-00032-FST, SRG2023-00011-FST, MYRG-GRG2023-00139-FST-UMDF, UMDF Professorial Fellowship of Mathematics, Macao SAR FDCT 0003/2023/RIA1 and Macao SAR FDCT 0024/2023/RIB1. 
The research of Xie is partially supported by  NSFC grants 12571238 and 12426203.


\bibliographystyle{plain}

\end{document}